\documentclass[12pt]{amsart}
\usepackage{amscd, amsfonts, amssymb, amsthm}
\usepackage{parskip, fullpage, verbatim}
\usepackage[colorlinks, breaklinks, linkcolor=blue]{hyperref} 
\usepackage{breakurl}
\usepackage{array}
\usepackage{booktabs}
\usepackage{wasysym} 

\newcommand\CA{{\mathcal A}}

\newcommand\CIF{{\mathcal {IF}}} 
 
\newcommand\CIFAC{{\mathcal {IF\!AC}}}

\newcommand\R{{\varrho}}

\newcommand\BBC{{\mathbb C}}
\newcommand\BBK{{\mathbb K}}

\newcommand {\GAP}{\textsf{GAP}}  
\newcommand {\CHEVIE}{\textsf{CHEVIE}}  
  
\newcommand {\Sage}{\textsf{SAGE}}

\newcommand\codim{\operatorname{codim}}

\newcommand\Fix{{\operatorname{Fix}}}

\newcommand\GL{\operatorname{GL}}

\newcommand\Poin{{\operatorname{Poin}}}

\numberwithin{equation}{section}

\theoremstyle{plain}
\newtheorem{lemma}[equation]{Lemma}
\newtheorem{theorem}[equation]{Theorem}

\newtheorem{corollary}[equation]{Corollary}
\newtheorem{proposition}[equation]{Proposition}
\theoremstyle{definition}
\newtheorem{defn}[equation]{Definition}
\newtheorem{remark}[equation]{Remark}

\thanks{We acknowledge 
support from the DFG-priority program 
SPP1489 ``Algorithmic and Experimental Methods in
Algebra, Geometry, and Number Theory''.}


\begin{document}

\title[Nice reflection arrangements]
{Nice reflection arrangements}

\author[T. Hoge]{Torsten Hoge}
\address
{Institut f\"ur Algebra, Zahlentheorie und Diskrete Mathematik,
Fakult\"at f\"ur Mathematik und Physik,
Leibniz Universit\"at Hannover,
Welfengarten 1,
30167 Hannover, Germany}
\email{hoge@math.uni-hannover.de}

\author[G. R\"ohrle]{Gerhard R\"ohrle}
\address
{Fakult\"at f\"ur Mathematik,
Ruhr-Universit\"at Bochum,
D-44780 Bochum, Germany}
\email{gerhard.roehrle@rub.de}


\allowdisplaybreaks

\begin{abstract}
The aim of this note is a classification of all 
nice and all inductively factored 
reflection arrangements.
It turns out that 
apart from the supersolvable instances
only the monomial groups 
$G(r,r,3)$ for $r \ge 3$ give rise to 
nice reflection arrangements.
As a consequence of this and 
of the classification of all inductively 
free reflection arrangements from 
\cite[Thm.\ 1.1]{hogeroehrle:indfree}, 
we deduce that the class of all 
inductively factored reflection arrangements 
coincides with the
supersolvable reflection arrangements.

Moreover, we extend these classifications to 
hereditarily factored and  hereditarily inductively 
factored  reflection arrangements.
\end{abstract}

\maketitle


\section{Introduction}

Let $\BBK$ be a field and let $V = \BBK^\ell$.
Let $\CA = (\CA,V)$ be a central $\ell$-arrangement of hyperplanes in $V$ and 
let $L(\CA)$ be its intersection lattice.
Let $\pi = (\pi_1, \ldots, \pi_s)$ be a partition of $\CA$.
Then $\pi$ is called \emph{nice} for $\CA$ or a \emph{factorization} of $\CA$ 
if roughly speaking it partitions $\CA$
into mutually linearly independent sets and these are compatible with the 
intersection lattice $L(\CA)$, see Definition \ref{def:factored}
below.

In 1992, Terao \cite{terao:factored}
introduced the notion of a \emph{nice} or \emph{factored} arrangement
to provide a necessary and sufficient condition for the 
Orlik-Solomon Algebra $A(\CA)$ of $\CA$ to 
admit a tensor factorization as 
a graded $\BBK$-vector space.
More precisely,
let $[\pi_i]$ be the $\BBK$-subspace of $A(\CA)$ 
spanned by $1$ and the generators of $A(\CA)$ corresponding to  $\pi_i$.
In \cite[Thm.\ 2.8]{terao:factored}
(cf.\ \cite[Thm.\ 3.87]{orlikterao:arrangements}),
Terao proved that the map 
\[
\kappa : [\pi_1] \otimes \cdots \otimes [\pi_s]  \to  A(\CA) 
\]
given by multiplication is an  isomorphism of graded $\BBK$-vector spaces
if and only if $\pi$ is nice for $\CA$,
see Theorem \ref{thm:teraofactored} below.
As a consequence, if $\pi = (\pi_1, \ldots, \pi_s)$  is nice for 
$\CA$, then $s = r$, the rank of $\CA$, and 
the Poincar\'e polynomial of the Orlik-Solomon algebra $A(\CA)$ of $\CA$ factors into linear terms
as follows:
\[
\Poin(A(\CA),t) = \prod_{i=1}^r (1 + |\pi_i| t),
\]
\cite{terao:factored}
(cf.\ \cite[Cor.\ 3.88]{orlikterao:arrangements}).
Note that
if $\CA$ is free, then the Poincar\'e polynomial 
$\pi(\CA,t)$ of $L(\CA)$
factors into linear terms 
as follows:
\[
\pi(\CA,t) = \prod_{i=1}^\ell (1 + b_i t),
\]
where $\exp \CA = \{b_1, \ldots, b_\ell\}$ are the exponents of $\CA$,
\cite[Thm.\ 4.137]{orlikterao:arrangements}, see Theorem \ref{thm:freefactors}
below.
It is natural to pose the question whether every nice arrangement is
free, \cite{terao:factored}.  This however is not the case;
likewise,  a free arrangement need not be factored
in general, \cite{terao:factored}.

In \cite[Thm.\ 1.5]{hogeroehrle:factored}, we gave an 
analogue of Terao's celebrated
addition-deletion theorem for 
free arrangements for the class of 
nice arrangements, see Theorem \ref{thm:add-del-factored} below.

Terao  \cite{terao:factored} showed that 
every supersolvable arrangement is factored, 
see Proposition \ref{prop:ssfactored}.
Indeed, every supersolvable arrangement is inductively factored,
see Proposition \ref{prop:ssindfactored}.
Moreover, Jambu and Paris showed that each inductively factored 
arrangement is inductively free, see Proposition \ref{prop:indfactoredindfree}
(\cite[Prop.\ 2.2]{jambuparis:factored}). 
Each of these classes of arrangements is properly contained in the other,
see \cite[Rem.\ 3.32]{hogeroehrle:factored}.

Suppose that $W$ is a finite, unitary
reflection group acting on the complex 
vector space $V$.
Let $\CA(W) = (\CA(W),V)$ be the associated 
hyperplane arrangement of $W$.
We refer to $\CA(W)$ as a 
\emph{reflection arrangement}.
The aim of this paper is to 
classify all factored and all inductively factored
reflection arrangements $\CA(W)$. 

In view of the aforementioned containments,
we first recall the classifications of 
the inductively free and 
the supersolvable reflection arrangements, 
from \cite[Thm.\ 1.1]{hogeroehrle:indfree}
and \cite[Thm.\ 1.2]{hogeroehrle:super}, respectively.
Here and later on we use the classification and 
labelling of the irreducible 
unitary reflection groups due to
Shephard and Todd, \cite{shephardtodd}. 

\begin{theorem}
\label{thm:indfreerefl}
For $W$ a finite complex reflection group,  
the reflection arrangement $\CA(W)$ of $W$ is 
inductively free if and only if 
$W$ does not admit an irreducible factor
isomorphic to a monomial group 
$G(r,r,\ell)$ for $r, \ell \ge 3$, 
$G_{24}, G_{27}, G_{29}, G_{31}, G_{33}$, or $G_{34}$.
\end{theorem}

The case for Coxeter groups in Theorem \ref{thm:indfreerefl}
is due to Barakat and Cuntz \cite{cuntz:indfree}.

\begin{theorem}
\label{thm:super}
For $W$ a finite
complex reflection group,  
$\CA(W)$ is 
supersolvable if and only if 
any irreducible factor of $W$ is
of rank at most $2$, 
is isomorphic either to 
a Coxeter group 
of type $A_\ell$ or $B_\ell$ for $\ell \ge 3$, or to a  
monomial group $G(r,p,\ell)$
for $r, \ell \ge 3$  and $p \ne r$.
\end{theorem}

We can now state our main classification results. 
Thanks to Proposition \ref{prop:product-factored},
the question whether $\CA$ is
nice reduces to the case when $\CA$ is irreducible.
Therefore, we may assume that $W$ is irreducible.
In view of Theorem \ref{thm:super} and 
Proposition \ref{prop:ssindfactored}, we can state our classification
results as follows:

\begin{theorem}
\label{thm:factoredrefl}
For $W$ a finite, irreducible, 
complex reflection group,  
$\CA(W)$ is nice if and only if 
either 
$\CA(W)$ is supersolvable or 
$W =  G(r,r,3)$ for $r \ge 3$.
\end{theorem}

Thanks to Proposition \ref{prop:product-indfactored},
the question whether $\CA$ is
inductively factored reduces to the case when $\CA$ is irreducible.
The classification of the inductively factored reflection arrangements is 
thus an immediate consequence of 
Theorems  \ref{thm:indfreerefl}, \ref{thm:factoredrefl} and 
Proposition 
\ref{prop:indfactoredindfree}.

\begin{corollary}
\label{cor:indfactoredrefl}
For $W$ a finite, complex reflection group,  $\CA(W)$
is inductively factored if and only if it is supersolvable. 
\end{corollary}

In contrast to Corollary \ref{cor:indfactoredrefl}, 
\cite[Ex.\ 3.19]{hogeroehrle:factored} shows that 
among restrictions of reflection arrangements 
there are inductively factored instances which are 
not supersolvable.  

A special case of a result due to Stanley implies that supersolvability of 
$\CA$ is inherited by 
all restrictions $\CA^X$, cf.\ \cite[Prop.\ 3.2]{stanley:super}.
While arbitrary inductively free arrangements
are not hereditary, cf.\ \cite[Ex.\ 2.16]{hogeroehrle:indfree}
and likewise for inductively factored arrangements, 
cf.\ \cite[Ex.\ 3.27]{hogeroehrle:factored},  
for reflection arrangements, inductive freeness is hereditary, 
\cite[Thm.\ 1.2]{hogeroehrle:indfree}.
Thus it is rather natural to ask whether the 
properties of being factored or inductively 
factored are also hereditary among reflection arrangements. 
These questions are answered in our next results: 

\begin{theorem}
\label{thm:heredfactored}
For $W$ a finite 
complex reflection group,  
$\CA(W)$ is factored
if and only if $\CA(W)$ is hereditarily factored.
\end{theorem}

The following is an easy consequence of 
Corollary \ref{cor:indfactoredrefl} and 
Stanley's result \cite[Prop.\ 3.2]{stanley:super}.

\begin{corollary}
\label{cor:indheredfactored}
For $W$ a finite 
complex reflection group,  
$\CA(W)$ is inductively factored
if and only if $\CA(W)$ is hereditarily inductively factored.
\end{corollary}

The paper is organized as follows.
Sections \ref{ssect:hyper} and \ref{ssect:os} 
recall basic 
notions and results for general hyperplane arrangements 
and their associated Orlik-Solomon algebras.
Subsequently, we recall relevant 
concepts of free, inductively free and supersolvable 
arrangements in Sections \ref{ssect:free} and \ref{ssect:super}.
All this is taken from \cite{orlikterao:arrangements}.
Sections \ref{ssect:factored} and \ref{ssect:indfactored}
revisit the concepts and main results on 
nice and inductively factored arrangements from 
\cite{terao:factored},
\cite{jambuparis:factored}, and \cite{hogeroehrle:factored}.
This is followed by a short recollection on 
hereditarily (inductively) factored arrangements 
in Section \ref{sect:heredfactored} from 
\cite{hogeroehrle:factored}.
In Section \ref{ssect:refl},
we discuss some required results on reflection arrangements.
Finally, Theorems \ref{thm:factoredrefl} and  
\ref{thm:heredfactored} 
are proved in Section \ref{sect:proof}.
 
For general information about arrangements 
and reflection groups we refer
the reader to 
\cite{orlikterao:arrangements},
\cite{bourbaki:groupes},
\cite{orliksolomon:unitaryreflectiongroups} and 
and \cite[\S 4, \S6]{orlikterao:arrangements}.

\section{Recollections and Preliminaries}

\subsection{Hyperplane Arrangements}
\label{ssect:hyper}
Let $V = \BBK^\ell$ 
be an $\ell$-dimensional $\BBK$-vector space.
A \emph{hyperplane arrangement} is a pair
$(\CA, V)$, where $\CA$ is a finite collection of hyperplanes in $V$.
Usually, we simply write $\CA$ in place of $(\CA, V)$.
We only consider central arrangements, i.e.\ $0 \in H$ for every $H \in \CA$.
We write $|\CA|$ for the number of hyperplanes in $\CA$.
The empty arrangement in $V$ is denoted by $\Phi_\ell$.

The \emph{lattice} $L(\CA)$ of $\CA$ is the set of subspaces of $V$ of
the form $H_1\cap \dotsm \cap H_r$ where $\{ H_1, \ldots, H_r \}$ is a subset
of $\CA$. 
For $X \in L(\CA)$, we have two associated arrangements, 
firstly the subarrangement 
$\CA_X :=\{H \in \CA \mid X \subseteq H\} \subseteq \CA$
of $\CA$ and secondly, 
the \emph{restriction of $\CA$ to $X$}, $(\CA^X,X)$, where 
$\CA^X := \{ X \cap H \mid H \in \CA \setminus \CA_X\}$.
Note that $V$ belongs to $L(\CA)$
as the intersection of the empty 
collection of hyperplanes and $\CA^V = \CA$. 
The lattice $L(\CA)$ is a partially ordered set by reverse inclusion:
$X \le Y$ provided $Y \subseteq X$ for $X,Y \in L(\CA)$.
We have a \emph{rank} function on $L(\CA)$: $r(X) := \codim_V(X)$.
The \emph{rank} $r(\CA)$ of $\CA$ is the rank of a maximal element in $L(\CA)$ with respect
to the partial order.
With this definition $L(\CA)$ is a geometric lattice, 
\cite[p.\ 24]{orlikterao:arrangements}.
Let $T_\CA = \cap_{H \in \CA} H$ be the center of $L(\CA)$.
If $\CA$ is \emph{central}, then $0 \in T_\CA$.
The $\ell$-arrangement $\CA$ is called \emph{essential} provided $r(\CA) = \ell$.
If $\CA$ is essential and central, then $T_\CA =\{0\}$.

Let $\CA$ be central and let
$X, Y \in L(\CA)$ with $X < Y$. 
We recall the following sublattices of $L(\CA)$ from 
\cite[Def.\ 2.10]{orlikterao:arrangements},
$L(\CA)_Y :=\{ Z \in L(\CA) \mid Z \le Y\}$, 
$L(\CA)^X :=\{ Z \in L(\CA) \mid X \le Z\}$,
and the \emph{interval}
$[X,Y] := L(\CA)_Y \cap L(\CA)^X = \{ Z \in L(\CA) \mid X \le Z \le Y\}$.

For $\CA \ne \Phi_\ell$, 
let $H_0 \in \CA$.
Define $\CA' := \CA \setminus\{ H_0\}$,
and $\CA'' := \CA^{H_0} = \{ H_0 \cap H \mid H \in \CA'\}$.
Then $(\CA, \CA', \CA'')$ is a \emph{triple} of arrangements,
\cite[Def.\ 1.14]{orlikterao:arrangements}.

The \emph{product}
$\CA = (\CA_1 \times \CA_2, V_1 \oplus V_2)$ 
of two arrangements $(\CA_1, V_1), (\CA_2, V_2)$
is defined by
\begin{equation*}
\label{eq:product}
\CA = \CA_1 \times \CA_2 := \{H_1 \oplus V_2 \mid H_1 \in \CA_1\} \cup 
\{V_1 \oplus H_2 \mid H_2 \in \CA_2\},
\end{equation*}
see \cite[Def.\ 2.13]{orlikterao:arrangements}.

Note that 
$\CA \times \Phi_0 = \CA$
for any arrangement $\CA$. 
If $\CA$ is of the form $\CA = \CA_1 \times \CA_2$, where 
$\CA_i \ne \Phi_0$ for $i=1,2$, then $\CA$
is called \emph{reducible}, else $\CA$
is \emph{irreducible}, 
\cite[Def.\ 2.15]{orlikterao:arrangements}.

\subsection{The Orlik-Solomon Algebra of an Arrangement}
\label{ssect:os}

The most basic algebraic invariant associated with an arrangement $\CA$ is its
so called \emph{Orlik-Solomon algebra} $A(\CA)$, introduced by Orlik and Solomon in \cite{orliksolomon:hyperplanes}.
The $\BBK$-algebra $A(\CA)$ is a graded and anti-commutative.
It is generated by $1$ in degree $0$ and by 
a set of degree $1$ generators $\{a_H \mid  H \in \CA\}$, e.g.\ see 
\cite[\S 3.1]{orlikterao:arrangements}.
Let $A(\CA) = \oplus_{i=0}^r A(\CA)_i$ be the decomposition of
$A(\CA)$ into homogeneous components, 
so that $\Poin(A(\CA), t) = \sum_{i=0}^r (\dim A(\CA)_i)t^i$,
where $r = r(\CA)$ is the rank of $\CA$.
In particular, $A(\CA)_0 = \BBK$ and  $A(\CA)_1 = \sum_{H\in \CA} \BBK a_H$.

Thanks to a fundamental result due to Orlik and Solomon
\cite[Thm.\ 2.6]{orliksolomon:hyperplanes} 
(cf.\ \cite[Thm.\ 3.68]{orlikterao:arrangements}),
the Poincar\'e polynomial of 
$A(\CA)$ 
coincides with the combinatorially defined Poincar\'e polynomial 
$\pi(\CA,t)$ of $\CA$,
\begin{equation*}
\label{eq:piA}
\Poin(A(\CA), t) = \pi(\CA,t).
\end{equation*}

The geometric significance of $A(\CA)$ stems 
from the fact that in case $\BBK = \BBC$ is the field of complex numbers,
Orlik and Solomon showed in 
\cite[Thm.\ 5.2]{orliksolomon:hyperplanes} 
that as an associative, graded $\BBC$-algebra $A(\CA)$
is isomorphic to 
the cohomology algebra of the complement $M(\CA)$ of 
the complex arrangement $\CA$ 
(cf.\ \cite[\S 5.4]{orlikterao:arrangements}):
\[
A(\CA) \cong H^*(M(\CA)).
\]
In particular, the Poincar\'e polynomial $\Poin(M(\CA), t)$ of $M(\CA)$
is given by $\Poin(A(\CA), t)$. 

Let $\pi = (\pi_1, \ldots, \pi_s)$ be a partition of $\CA$ and let 
\[
[\pi_i] := \BBK + \sum_{H\in \pi_i} \BBK a_H
\] 
be the $\BBK$-subspace of $A(\CA)$ 
spanned by $1$ and the set of $\BBK$-algebra generators $a_H$ 
of $A(\CA)$ corresponding to the members in $\pi_i$.
So the Poincar\'e polynomial of the graded $\BBK$-vector 
space $[\pi_i]$ is just  
$\Poin([\pi_i],t) = 1 + | \pi_i| t$.
Consider the canonical $\BBK$-linear map 
\begin{equation}
\label{eq:factoredOS}
\kappa : [\pi_1] \otimes \cdots \otimes [\pi_s] \to A(\CA)
\end{equation}
given by multiplication. We say that $\pi$ gives
rise to a \emph{tensor factorization} of $A(\CA)$ 
if $\kappa$ is an isomorphism of graded $\BBK$-vector spaces. In this case
$s = r$, as $r$ is the top degree of $A(\CA)$, and thus we get 
a factorization of the Poincar\'e polynomial of $A(\CA)$ into linear terms
\begin{equation*}
\label{eq:poinOS}
\Poin(A(\CA),t) = \prod_{i=1}^r (1 + |\pi_i| t).
\end{equation*}
For $\CA = \Phi_\ell$ the empty arrangement, 
we set $[\varnothing] := \BBK$, so that 
$\kappa : [\varnothing] \cong A(\Phi_\ell)$.

In \cite[Thm.\ 5.3]{orliksolomonterao:hyperplanes},
Orlik, Solomon and Terao showed that 
a supersolvable arrangement $\CA$ admits a partition 
$\pi$  which gives rise to a tensor factorization of $A(\CA)$
via $\kappa$  in \eqref{eq:factoredOS} 
(cf.\ \cite[Thm.\ 3.81]{orlikterao:arrangements});
see Proposition \ref{prop:ssfactored} below.

In \cite[Thm.\ 2.8]{terao:factored},
Terao proved that 
$\pi$ gives
rise to a tensor factorization of the Orlik-Solomon algebra $A(\CA)$
via $\kappa$ as in \eqref{eq:factoredOS}
if and only if 
$\pi$ is 
\emph{nice} for $\CA$ (Definition \ref{def:factored}), 
see Theorem \ref{thm:teraofactored}
(cf.\ \cite[Thm.\ 3.87]{orlikterao:arrangements}).
So nice arrangements are a generalization of 
supersolvable ones.
Note that $\kappa$ is not an isomorphism of $\BBK$-algebras.

\subsection{Free  and inductively free Arrangements}
\label{ssect:free}

Free arrangements play a crucial role in the theory of arrangements.
See \cite[\S 4]{orlikterao:arrangements} for the definition and 
basic properties of free arrangements. If $\CA$ is free, then 
we can associate with $\CA$ the multiset of its \emph{exponents}, 
denoted $\exp \CA$. 

Owing to \cite[Prop.\ 4.28]{orlikterao:arrangements}, 
free arrangements behave well with respect to 
the  product construction for arrangements.

\begin{proposition}
\label{prop:product-free}
Let $\CA_1, \CA_2$ be two arrangements.
Then  $\CA = \CA_1 \times \CA_2$ is free
if and only if both 
$\CA_1$ and $\CA_2$ are free and in that case
the multiset of exponents of $\CA$ is given by 
$\exp \CA = \{\exp \CA_1, \exp \CA_2\}$.
\end{proposition}

Terao's celebrated \emph{Addition-Deletion Theorem} 
\cite{terao:freeI} plays a 
pivotal role in the study of free arrangements, 
\cite[Thm.\ 4.51]{orlikterao:arrangements}.

\begin{theorem}
\label{thm:add-del}
Suppose that $\CA \ne \Phi_\ell$.
Let  $(\CA, \CA', \CA'')$ be a triple of arrangements. Then any 
two of the following statements imply the third:
\begin{itemize}
\item[(i)] $\CA$ is free with $\exp \CA = \{ b_1, \ldots , b_{\ell -1}, b_\ell\}$;
\item[(ii)] $\CA'$ is free with $\exp \CA' = \{ b_1, \ldots , b_{\ell -1}, b_\ell-1\}$;
\item[(iii)] $\CA''$ is free with $\exp \CA'' = \{ b_1, \ldots , b_{\ell -1}\}$.
\end{itemize}
\end{theorem}

Terao's \emph{Factorization Theorem}
\cite{terao:freefactors} shows
that the Poincar\'e polynomial 
$\pi(\CA,t)$ of a free arrangement $\CA$
factors into linear terms 
given by the exponents of $\CA$
(cf.\ \cite[Thm.\ 4.137]{orlikterao:arrangements}):

\begin{theorem}
\label{thm:freefactors}
Suppose that 
$\CA$ is free with $\exp \CA = \{ b_1, \ldots , b_\ell\}$.
Then 
\[
\pi(\CA,t) = \prod_{i=1}^\ell (1 + b_i t).
\]
\end{theorem}

Theorem \ref{thm:add-del} motivates the notion of 
\emph{inductively free} arrangements,   
\cite[Def.\ 4.53]{orlikterao:arrangements}.

\begin{defn}
\label{def:indfree}
The class $\CIF$ of \emph{inductively free} arrangements 
is the smallest class of arrangements subject to
\begin{itemize}
\item[(i)] $\Phi_\ell \in \CIF$ for each $\ell \ge 0$;
\item[(ii)] if there exists a hyperplane $H_0 \in \CA$ such that both
$\CA'$ and $\CA''$ belong to $\CIF$, and $\exp \CA '' \subseteq \exp \CA'$, 
then $\CA$ also belongs to $\CIF$.
\end{itemize}
\end{defn}

In \cite[Prop.\ 2.10]{hogeroehrle:indfree},
we showed that the compatibility 
of products with free arrangements
from Proposition \ref{prop:product-free}
restricts to inductively free arrangements.

\begin{proposition}
\label{prop:product-indfree}
Let $\CA_1, \CA_2$ be two arrangements.
Then  $\CA = \CA_1 \times \CA_2$ is inductively free
if and only if both 
$\CA_1$ and $\CA_2$ are inductively free and in that case
$\exp \CA = \{\exp \CA_1, \exp \CA_2\}$.
\end{proposition}

\subsection{Supersolvable Arrangements}
\label{ssect:super}

Let $\CA$ be an arrangement.
Following \cite[\S 2]{orlikterao:arrangements}, we say
that $X \in L(\CA)$ is \emph{modular}
provided $X + Y \in L(\CA)$ for every $Y \in L(\CA)$, 
cf.\ \cite[Def.\ 2.32, Cor.\ 2.26]{orlikterao:arrangements}.

The following notion is due to Stanley \cite{stanley:super}. 

\begin{defn}
\label{def:super}
Let $\CA$ be a central (and essential) $\ell$-arrangement.
We say that $\CA$ is \emph{supersolvable} 
provided there is a maximal chain
\[
V = X_0 < X_1 < \ldots < X_{\ell-1} < X_\ell = \{0\}
\]
of modular elements $X_i$ in $L(\CA)$.
\end{defn}

\begin{remark}
\label{rem:2-arr}
By \cite[Ex.\ 2.28]{orlikterao:arrangements}, 
$V$, $\{0\}$ and the members in $\CA$ 
are always modular in $L(\CA)$.
It follows  that all $0$- $1$-, and $2$-arrangements are supersolvable.
\end{remark}

Note that supersolvable arrangements
are always inductively free, 
\cite[Thm.\ 4.58]{orlikterao:arrangements}.

In \cite[Prop.\ 2.6]{hogeroehrle:super},
we showed that the compatibility 
of products with inductively free arrangements
from Proposition \ref{prop:product-indfree}
restricts further to supersolvable arrangements.

\subsection{Nice Arrangements}
\label{ssect:factored}

The notion of a \emph{nice} or \emph{factored} 
arrangement goes back to Terao \cite{terao:factored}.
It generalizes the concept of a supersolvable arrangement, see
Proposition \ref{prop:ssfactored}.
We recall the relevant notions and results from \cite{terao:factored} 
(cf.\  \cite[\S 2.3]{orlikterao:arrangements}).

\begin{defn}
\label{def:independent}
Let $\pi = (\pi_1, \ldots , \pi_s)$ be a partition of $\CA$.
Then $\pi$ is called \emph{independent}, provided 
for any choice $H_i \in \pi_i$ for $1 \le i \le s$,
the resulting $s$ hyperplanes are linearly independent, i.e.\
$r(H_1 \cap \ldots \cap H_s) = s$.
\end{defn}

\begin{defn}
\label{def:indpart}
Let $\pi = (\pi_1, \ldots , \pi_s)$ be a partition of $\CA$
and let $X \in L(\CA)$.
The \emph{induced partition} $\pi_X$ of $\CA_X$ is given by the non-empty 
blocks of the form $\pi_i \cap \CA_X$.
\end{defn}

\begin{defn}
\label{def:factored}
The partition 
$\pi$ of $\CA$ is
\emph{nice} for $\CA$ or a \emph{factorization} of $\CA$  provided 
\begin{itemize}
\item[(i)] $\pi$ is independent, and 
\item[(ii)] for each $X \in L(\CA) \setminus \{V\}$, the induced partition $\pi_X$ admits a block 
which is a singleton. 
\end{itemize}
If $\CA$ admits a factorization, then we also say that $\CA$ is \emph{factored} or \emph{nice}.
\end{defn}

\begin{remark}
\label{rem:factored}
(i). 
Vacuously, the empty partition is nice for the empty arrangement $\Phi_\ell$.

(ii).
If $\CA \ne \Phi_\ell$,  
$\pi$ is a nice partition of $\CA$
and $X \in L(\CA)\setminus\{V\}$, then the non-empty parts of the 
induced partition $\pi_X$ form a nice partition of $\CA_X$.
For, if $\pi$ is independent, then clearly so is $\pi_X$.
Moreover, 
if $Y \in L(\CA_X)\setminus\{V\}$, then $Y \in L(\CA)$ and since 
$\pi$ is a factorization of $\CA$, there is a block $\pi_i$ of $\pi$
such that $\pi_i \cap \CA_Y =\{H\}$ is a singleton.
Since $X \subseteq Y \subseteq H$, $\pi_i \cap \CA_X$ is a 
non-empty block of $\pi_X$ so that 
$(\pi_i \cap \CA_X) \cap (\CA_X)_Y  = \pi_i \cap \CA_Y$ is a singleton
(note that for $X \subseteq Y$ in $L(\CA)$, we have $(\CA_X)_Y = \CA_Y$).

(iii). 
Since the singleton condition in Definition \ref{def:factored}(ii)
also applies to the center $T_\CA$ of $L(\CA)$, 
a factorization $\pi$ of $\CA \ne \Phi_\ell$
always admits a singleton as one of its parts.
Also note that for a hyperplane, 
the singleton condition trivially holds.
\end{remark}

We recall the main results from \cite{terao:factored} 
(cf.\  \cite[\S 3.3]{orlikterao:arrangements}) that motivated 
Definition \ref{def:factored}.

\begin{theorem}
\label{thm:teraofactored}
Let $\CA$ be a central $\ell$-arrangement and let 
$\pi = (\pi_1, \ldots, \pi_s)$ be a partition of $\CA$.
Then the $\BBK$-linear map $\kappa$ 
defined in \eqref{eq:factoredOS}
is an isomorphism of graded $\BBK$-vector spaces
if and only if $\pi$ is nice for $\CA$.
\end{theorem}

\begin{corollary}
\label{cor:teraofactored}
Let  $\pi = (\pi_1, \ldots, \pi_s)$ be a factorization of $\CA$.
Then the following hold:
\begin{itemize}
\item[(i)] $s = r = r(\CA)$ and 
\[
\Poin(A(\CA),t) = \prod_{i=1}^r (1 + |\pi_i|t);
\]
\item[(ii)]
the multiset $\{|\pi_1|, \ldots, |\pi_r|\}$ only depends on $\CA$;
\item[(iii)]
for any $X \in L(\CA)$, we have
\[
r(X) = |\{ i \mid \pi_i \cap \CA_X \ne \varnothing \}|.
\]  
\end{itemize}
\end{corollary}

\begin{remark}
\label{rem:factorediscombinatorial}
It follows from \eqref{eq:piA}
and Corollary \ref{cor:teraofactored} that 
the question whether $\CA$ is factored is a purely combinatorial 
property and only depends on the lattice $L(\CA)$. 
\end{remark}

\begin{remark}
\label{rem:exponents}
Suppose that $\CA$ is free of rank $r$.
Then $\CA = \Phi_{\ell-r} \times \CA_0$, 
where $\CA_0$ is an essential, free $r$-arrangement
(cf.\ \cite[\S 3.2]{orlikterao:arrangements}), and so,  
thanks to Proposition \ref{prop:product-free}, 
$\exp \CA = \{0^{\ell-r}, \exp \CA_0\}$.
Suppose that $\pi = (\pi_1, \ldots, \pi_r)$ is a nice partition 
of $\CA$. 
Then by the factorization properties of the Poincar\'e polynomials
for free and factored arrangements, 
Theorem \ref{thm:freefactors}, respectively 
Corollary \ref{cor:teraofactored}(i) and 
\eqref{eq:piA}
we have 
\begin{equation*}
\label{eq:exp}
\exp \CA = \{0^{\ell-r}, |\pi_1|, \ldots, |\pi_r|\}.
\end{equation*}
In particular, if $\CA$ is essential, then 
\begin{equation*}
\label{eq:ess-exp}
\exp \CA = \{|\pi_1|, \ldots, |\pi_\ell|\}.
\end{equation*}
\end{remark}

Finally, we record 
\cite[Ex.\ 2.4]{terao:factored},
which shows that nice arrangements generalize 
supersolvable ones 
(cf.\  \cite[Thm.\ 5.3]{orliksolomonterao:hyperplanes}, 
\cite[Prop.\ 3.2.2]{jambu:factored}, 
\cite[Prop.\ 2.67, Thm.\ 3.81]{orlikterao:arrangements}).

\begin{proposition}
\label{prop:ssfactored}
Let $\CA$ be a central,  supersolvable arrangement of rank $r$.
Let 
\[
V = X_0 < X_1 < \ldots < X_{r-1} < X_r = T_\CA
\]
be a  maximal chain of modular elements in $L(\CA)$.
Define
$\pi_i = \CA_{X_i} \setminus \CA_{X_{i-1}}$
for $1 \le i \le r$.
Then $\pi = (\pi_1, \ldots, \pi_r)$ is a nice partition of $\CA$.
In particular, the $\BBK$-linear map 
$\kappa$ defined in \eqref{eq:factoredOS} is an 
isomorphism of graded $\BBK$-vector spaces.
\end{proposition}

\subsection{Inductively factored Arrangements}
\label{ssect:indfactored}

Following Jambu and Paris 
\cite{jambuparis:factored}, 
we introduce further notation.

\begin{defn} 
\label{def:distinguished}
Suppose $\CA \ne \Phi_\ell$. 
Let $\pi = (\pi_1, \ldots, \pi_s)$ be a partition of $\CA$.
Let $H_0 \in \pi_1$ and 
let $(\CA, \CA', \CA'')$ be the triple associated with $H_0$. 
We say that $H_0$ is \emph{distinguished (with respect to $\pi$)}
provided $\pi$ induces a factorization $\pi'$ of 
$\CA'$, i.e.\ the non-empty 
subsets $\pi_i \cap \CA'$ form a nice partition
of $\CA'$. Note that since $H_0 \in \pi_1$, we have
$\pi_i \cap \CA' = \pi_i\not= \varnothing$ 
for $i = 2, \ldots, s$. 

Also, associated with $\pi$ and $H_0$, we define 
the \emph{restriction map}
\[
\R := \R_{\pi,H_0} : \CA \setminus \pi_1 \to \CA''\ \text{ given by } \ H \mapsto H \cap H_0
\]
and set 
\[
\pi_i'' := \R(\pi_i) = \{H \cap H_0 \mid H \in \pi_i\} \
\text{ for }\  2 \le i \le s.
\]

In general $\R$ need not be surjective nor injective.
However, since we are only concerned with cases when 
$\pi'' = (\pi_2'', \ldots, \pi_s')$ is a
partition of $\CA''$,  
$\R$ has to be onto and 
$\R(\pi_i) \cap \R(\pi_j) = \varnothing$ for $i \ne j$.
As we have observed in \cite{hogeroehrle:factored},
the natural condition in this context is the injectivity of $\R$.
\end{defn}

In \cite[Thm.\ 1.5, Thm.\ 1.7]{hogeroehrle:factored}, we gave the following 
analogues of Terao's 
Addition-Deletion Theorem \ref{thm:add-del} for 
free arrangements for the class of 
nice arrangements.

\begin{theorem}
\label{thm:add-del-factored}
Suppose that $\CA \ne \Phi_\ell$.
Let $\pi = (\pi_1, \ldots, \pi_\ell)$ be a  partition  of $\CA$.
Let $H_0 \in \pi_1$ and let
$(\CA, \CA', \CA'')$ be the triple associated with $H_0$. 
Suppose that $\R: \CA \setminus \pi_1 \to \CA''$ 
is injective.
Then any two of the following statements imply the third:
\begin{itemize}
\item[(i)] $\pi$ is nice for $\CA$;
\item[(ii)] $\pi'$ is nice for $\CA'$;
\item[(iii)] $\pi''$ is nice for $\CA''$.
\end{itemize}
\end{theorem}

As indicated above, nice arrangements need not be free 
and vice versa, as observed by 
Terao \cite{terao:factored}.
Combining Theorem \ref{thm:add-del-factored} with Terao's
Addition-Deletion Theorem \ref{thm:add-del} for free arrangements, 
in \cite[Thm.\ 1.7]{hogeroehrle:factored}
we obtain the following 
Addition-Deletion Theorem for 
the proper subclass of 
free and nice arrangements.

\begin{theorem}
\label{thm:add-del-free-factored}
Suppose that $\CA \ne \Phi_\ell$.
Let $\pi = (\pi_1, \ldots, \pi_\ell)$ be a  partition  of $\CA$.
Let $H_0 \in \pi_\ell$ and 
let $(\CA, \CA', \CA'')$ be the triple associated with $H_0$. 
Suppose that $\R: \CA \setminus \pi_1 \to \CA''$ 
is injective. 
Then any two of the following statements imply the third:
\begin{itemize}
\item[(i)] $\pi$ is nice for $\CA$ and $\CA$ is free; 
\item[(ii)] $\pi'$ is nice for $\CA'$ and $\CA'$ is free; 
\item[(iii)] $\pi''$ is nice for $\CA''$ and $\CA''$ is free. 
\end{itemize}
\end{theorem}

Worth noting is the fact that
in Theorem \ref{thm:add-del-free-factored}
we do not need to explicitly require the 
containment conditions on the sets of exponents of the arrangements involved,
see Theorem \ref{thm:add-del}. This is a consequence of the
presence of the underlying factorizations.

We also note that the injectivity condition on $\R$ 
in both theorems is necessary, 
see \cite[Ex.\ 3.3, Ex.\ 3.20]{hogeroehrle:factored}.

The Addition-Deletion Theorem \ref{thm:add-del-factored} 
for nice arrangements motivates
the following stronger notion of factorization, 
cf.\ \cite{jambuparis:factored}, \cite[Def.\ 3.8]{hogeroehrle:factored}.

\begin{defn} 
\label{def:indfactored}
The class $\CIFAC$ of \emph{inductively factored} arrangements 
is the smallest class of pairs $(\CA, \pi)$ of 
arrangements $\CA$ together with a partition $\pi$
subject to
\begin{itemize}
\item[(i)] $(\Phi_\ell, (\varnothing)) \in \CIFAC$ for each $\ell \ge 0$;
\item[(ii)] if there exists a partition $\pi$ of $\CA$ 
and a hyperplane $H_0 \in \pi_1$ such that 
for the triple $(\CA, \CA', \CA'')$ associated with $H_0$ 
the restriction map $\R = \R_{\pi, H_0} : \CA \setminus \pi_1 \to \CA''$ 
is injective and for the induced partitions $\pi'$ of $\CA'$ and 
$\pi''$ of $\CA''$ 
both $(\CA', \pi')$ and $(\CA'', \pi'')$ belong to $\CIFAC$, 
then $(\CA, \pi)$ also belongs to $\CIFAC$.
\end{itemize}
If $(\CA, \pi)$ is in $\CIFAC$, then we say that
$\CA$ is \emph{inductively factored with respect to $\pi$}, or else
that $\pi$ is an \emph{inductive factorization} of $\CA$. 
Sometimes, we simply say $\CA$ is \emph{inductively factored} without 
reference to a specific inductive factorization of $\CA$.
\end{defn}

Our definition of inductively factored arrangements
in Definition \ref{def:indfactored}
differs slightly from the one by 
Jambu and Paris \cite{jambuparis:factored};
see \cite[Rem.\ 3.9]{hogeroehrle:factored}.

\begin{defn}
\label{def:heredindfactored}
In analogy to hereditary freeness and hereditary inductive freeness, 
\cite[Def.\ 4.140, p.\ 253]{orlikterao:arrangements}, 
we say that $\CA$ is \emph{hereditarily factored}
provided $\CA^X$ is factored for every $X \in L(\CA)$
and that $\CA$ is \emph{hereditarily inductively factored}
provided $\CA^X$ is inductively factored for every $X \in L(\CA)$.
\end{defn}

In \cite[Prop.\ 3.11]{hogeroehrle:factored}, we 
strengthened Proposition \ref{prop:ssfactored} as follows.

\begin{proposition} 
\label{prop:ssindfactored}
If $\CA$ is supersolvable, then $\CA$ is inductively factored.
\end{proposition}

\begin{remark}
\label{rem:dim2}
Since any $1$- and $2$-arrangement is supersolvable,
Remark \ref{rem:2-arr}, 
each such is inductively factored, by 
Proposition \ref{prop:ssindfactored}.
\end{remark}

In \cite[Prop.\ 2.2]{jambuparis:factored},
Jambu and Paris showed that 
inductively factored arrangements are always inductively free;
see also \cite[Prop.\ 3.14]{hogeroehrle:factored}.
(Jambu and Paris only claimed freeness but their 
proof actually does give the stronger result.)

\begin{proposition}
\label{prop:indfactoredindfree}
Let $\pi = (\pi_1, \ldots, \pi_r)$ be an inductive factorization of $\CA$. 
Then $\CA$ is inductively free with exponents 
$\exp \CA = \{0^{\ell-r}, |\pi_1|, \ldots, |\pi_r|\}$.
\end{proposition}

\begin{remark}
\label{rem:d4}
The converse of Proposition \ref{prop:indfactoredindfree} is false,
Terao has already noted that the reflection arrangement $\CA(D_4)$  of
the Coxeter group of type $D_4$ is not factored.
But $\CA(D_4)$ is inductively free, \cite[Ex.\ 2.6]{jambuterao:free}.
\end{remark}

\begin{remark}
\label{rem:g333}
Jambu and Paris observed that the reflection arrangement
$\CA(G(3,3,3))$ 
of the complex reflection group $G(3,3,3)$ 
is factored but not inductively factored 
\cite{jambuparis:factored}.
Note that  $\CA(G(3,3,3))$ is not 
inductively free, 
\cite[Thm.\ 1.1]{hogeroehrle:indfree}.
In particular, a free and factored arrangement need not be 
inductively factored; see also 
\cite[Ex.\ 3.20]{hogeroehrle:factored}.
Not even an inductively free and factored arrangement need be
inductively factored, see \cite[Ex.\ 3.22]{hogeroehrle:factored}.
\end{remark}

In \cite[Prop.\ 3.28]{hogeroehrle:factored},
we showed that
the product construction behaves well with factorizations.

\begin{proposition}
\label{prop:product-factored}
Let $\CA_1, \CA_2$ be two arrangements.
Then  $\CA = \CA_1 \times \CA_2$ is nice
if and only if both 
$\CA_1$ and $\CA_2$ are nice.
\end{proposition}

And in \cite[Prop.\ 3.29]{hogeroehrle:factored},
we strengthened  
Proposition \ref{prop:product-factored}
further by showing that 
the compatibility with products restricts 
to the class of inductively factored arrangements.

\begin{proposition}
\label{prop:product-indfactored}
Let $\CA_1, \CA_2$ be two arrangements.
Then  $\CA = \CA_1 \times \CA_2$ is 
inductively factored if and only if both 
$\CA_1$ and $\CA_2$ are 
inductively factored  and in that case
the multiset of exponents of $\CA$ is given by 
$\exp \CA = \{\exp \CA_1, \exp \CA_2\}$.
\end{proposition}

\subsection{Hereditarily factored Arrangements}
\label{sect:heredfactored}

While a $2$-arrangement is always inductively factored,
by Remark \ref{rem:2-arr},  
in general, a factored $3$-arrangement need not be inductively factored,
not even if it is a  reflection arrangement,
see Remark \ref{rem:g333}.
Nevertheless, for a $3$-arrangement, 
we have the following 
counterpart to \cite[Lem.\ 2.15]{hogeroehrle:indfree}
in our setting; see \cite[Lem.\ 3.26]{hogeroehrle:factored}.

\begin{lemma}
\label{lem:3-arr}
Suppose that $\ell = 3$. Then 
$\CA$ is (inductively) factored if and only if it is 
hereditarily (inductively) factored.
\end{lemma}

The compatibility from Propositions \ref{prop:product-factored} 
and \ref{prop:product-indfactored}
restricts further to the classes of hereditarily factored and 
hereditarily inductively factored arrangements, respectively;
see \cite[Cor.\ 3.31]{hogeroehrle:factored}.

\begin{corollary}
\label{cor:product-heredindfactored}
Let $\CA_1, \CA_2$ be two arrangements.
Then  $\CA = \CA_1 \times \CA_2$ is 
hereditarily (inductively) factored if and only if both 
$\CA_1$ and $\CA_2$ are 
hereditarily (inductively) factored.
In case of inductively factored arrangements
the multiset of exponents of $\CA$ is given by 
$\exp \CA = \{\exp \CA_1, \exp \CA_2\}$.
\end{corollary}

\subsection{Reflection Groups and Reflection Arrangements}
\label{ssect:refl}
The irreducible finite complex reflection groups were 
classified by Shephard and Todd, \cite{shephardtodd}.
Let $W  \subseteq \GL(V)$ be a finite complex reflection group.
For $w \in W$, we write 
$\Fix(w) :=\{ v\in V \mid w v = v\}$ for 
the fixed point subspace of $w$.
For $U \subseteq V$ a subspace, we 
define the \emph{parabolic subgroup}
$W_U$ of $W$ by 
$W_U := \{w \in W \mid U \subseteq \Fix(w)\}$.

The \emph{reflection arrangement} $\CA = \CA(W)$ of $W$ in $V$ is 
the hyperplane arrangement 
consisting of the reflecting hyperplanes of the elements in $W$
acting as reflections on $V$.
By Steinberg's Theorem \cite[Thm.\ 1.5]{steinberg:invariants},
for $U \subseteq V$ a subspace, 
the parabolic subgroup
$W_U$ is itself a complex reflection group,
generated by the unitary reflections in $W$ that are contained
in $W_U$. 
Thus, we identify the 
reflection arrangement $\CA(W_U)$
of $W_U$ as a subarrangement of $\CA$.

First we record a consequence of the classification of the
inductively free reflection arrangements,
Theorem \ref{thm:indfreerefl}:

\begin{corollary}
\label{cor:parabolicindfree}
Let $\CA = \CA(W)$ be an inductively free reflection
arrangement. Then $\CA(W_X)$ is also inductively free for any 
parabolic subgroup $W_X$  of $W$.
\end{corollary}

\begin{proof}
The result follows readily from 
Theorem \ref{thm:indfreerefl} and 
the classification of the parabolic subgroups of the 
irreducible complex reflection groups, 
\cite[\S 6.4, App.\ C]{orlikterao:arrangements}.
\end{proof}

Next we record an elementary but nevertheless
very useful inductive tool.

\begin{lemma}
\label{lem:parabolicfactored}
Let $\CA = \CA(W)$ be a nice reflection
arrangement. Then $\CA(W_X)$ is also nice for any 
parabolic subgroup $W_X$  of $W$.
\end{lemma}

\begin{proof}
Note that for $X \in L(\CA)$, we have 
$\CA(W_X) = \CA_X$, cf.\ 
\cite[Thm.\ 6.27, Cor.\ 6.28]{orlikterao:arrangements}.
The desired result follows from
Remark \ref{rem:factored}(ii). 
\end{proof}

Using Corollary \ref{cor:indfactoredrefl}, 
Lemma \ref{lem:parabolicfactored} restricts 
to inductively factored reflection arrangements:

\begin{corollary}
\label{cor:parabolicindfactored}
Let $\CA = \CA(W)$ be an inductively factored reflection
arrangement. Then $\CA(W_X)$ is also inductively factored for any 
parabolic subgroup $W_X$  of $W$.
\end{corollary}

\begin{proof}
As noted above, for $X \in L(\CA)$, we have 
$\CA(W_X) = \CA_X$, cf.\ 
\cite[Thm.\ 6.27, Cor.\ 6.28]{orlikterao:arrangements}.
Moreover, since 
$L(\CA_X) = L(\CA)_X = [V,X]$ is an interval in $L(\CA)$, 
\cite[Def.\ 2.10]{orlikterao:arrangements},
the result follows from 
Corollary \ref{cor:indfactoredrefl} and Stanley's
result \cite[Prop.\ 3.2]{stanley:super}.
\end{proof}

\section{Nice Reflection Arrangements}
\label{sect:proof}

In this section we provide proofs of Theorems \ref{thm:factoredrefl} and \ref{thm:heredfactored}.

\subsection{Proof of Theorem \ref{thm:factoredrefl}.}
\label{ss:findfactored}

It follows from 
Proposition \ref{prop:ssfactored}
that if $\CA$ is supersolvable, then it is nice.
Thus, by Proposition \ref{prop:product-factored},
we may assume that $W$ is irreducible so that 
$\CA(W)$ is not supersolvable.
Then in view of Theorem \ref{thm:super}, 
Theorem \ref{thm:factoredrefl}  follows from 
Lemmas \ref{lem:grr3} -- \ref{lem:g31} below.

\begin{lemma}
\label{lem:grr3}
Let $W = G(r,r,3)$ for $r \ge 3$. 
Then $\CA(W)$ is nice.
\end{lemma}

\begin{proof}
Let  $r \ge 3$, 
$W = G(r,r,3)$ and $\CA = \CA(W)$.
We denote the coordinate functions of ${\mathbb C}^3$ by $x,y$ and $z$. 
Furthermore, let $\zeta$  be a primitive $r$-th root of $1$.
There are three families of hyperplanes in $\CA$ given 
as follows. For $i = 0,\ldots, r-1$, we set
\[
A_i := \ker(x-\zeta^i y),  
B_i := \ker(x-\zeta^i z), 
\text{and} \quad  
C_i = \ker(y-\zeta^i z).
\] 
We also consider the hyperplanes
$A_i$, $B_i$, and $C_i$ for $i \in {\mathbb Z}$, 
simply by taking $i$ modulo $r$.

One readily checks that the members $X$ of rank $2$ of 
$L(\CA)$ are given by the following subarrangements $\CA_X$ of $\CA$
(i.e.\  $X = \cap_{H \in \CA_X} H$) :
\[
\{A_0, A_1, \ldots, A_{r-1}\},
\{B_0, B_1, \ldots, B_{r-1}\},
\{C_0, C_1, \ldots, C_{r-1}\},
\]
and
\[
\{A_i,  B_j,  C_k \mid i,j \in \{0,\ldots, r-1\} 
\text{ and } k = j-i \mod r \}.
\]

Note that $\exp \CA(G(r,r,3)) = \{1, r+1, 2(r-1)\}$,
cf.\ \cite[Cor.\ 6.86]{orlikterao:arrangements}.
We claim that the following partition $\pi$ is a factorization of $\CA$:
\[
\pi = (\pi_1, \pi_2 , \pi_3 ) 
:= (\{A_0\}, \{A_1, \ldots, A_{r-1}, B_0, C_0\}, 
\{B_1, \ldots , B_{r-1}, C_1, \ldots , C_{r-1} \}).
\]

First we check the singleton condition
from Definition \ref{def:factored}(ii).
Thanks to Remark \ref{rem:factored}(iii), we only need 
to check this for $X$ of rank $2$.

If $\CA_X = \{A_0, A_1, \ldots, A_{r-1}\}$, then
$\pi_1 \cap \CA_X = \{A_0\}$. 
Likewise, if 
$\CA_X = \{B_0, B_1, \ldots, B_{r-1}\}$,
respectively 
$\CA_X = \{C_0, C_1, \ldots, C_{r-1}\}$,
then 
$\pi_2 \cap \CA_X = \{B_0\}$, 
respectively 
$\pi_2 \cap \CA_X = \{C_0\}$, 
is the desired singleton.
If $\CA_X = \{A_0, B_j, C_{0-j}\}$, 
then $\pi_1 \cap \CA_X = \{A_0\}$. 
Finally, if $\CA_X = \{A_i, B_j, C_{i-j}\}$ for $i,j \not=0 \mod r$,
then $\pi_2 \cap \CA_X = \{A_i\}$ 
and if $\CA_X = \{A_i, B_0, C_i\}$ for $i \not=0 \mod r$,
then $\pi_3 \cap \CA_X = \{C_i\}$ is the desired singleton.

Finally, we need to show that $\pi$ is independent. 
Up to interchanging the roles of $B_j$ and $C_j$, 
we only have to check three cases:
If $Y_1 := A_0 \cap A_i \cap B_j$ where $i, j>0$,  
then we have
\[
r(Y_1)  = r(A_0 \cap \ldots \cap A_{r-1} \cap B_j) =
r(A_0 \cap \ldots \cap A_{r-1}) + 1 = 3.
\]
If $Y_2 := A_0 \cap B_0 \cap B_i$ for $i>0$,
then as in the case for $Y_1$, we see  
\[
r(Y_2)  = r(A_0 \cap B_0 \cap \ldots \cap B_{r-1}) =
r(B_0 \cap \ldots \cap B_{r-1}) + 1 = 3
\]
and likewise if 
$Y_3 := A_0 \cap C_0 \cap C_i$ for $i>0$,
then we get
\[
r(Y_3)  = r(A_0 \cap C_0 \cap \ldots \cap C_{r-1}) =
r(C_0 \cap \ldots \cap C_{r-1}) + 1 = 3.
\]
Finally, for $i>0$, it is immediate that  
$A_0 \cap B_0 \cap C_i = \{0\}$ has rank $3$. 
\end{proof}

Jambu and Paris \cite{jambuparis:factored} 
already observed that the reflection arrangement
$\CA(G(3,3,3))$ is factored but not inductively factored,
see Remark \ref{rem:g333} above. 
Note that  $\CA(G(3,3,3))$ is not 
inductively free, 
cf.\ Theorem \ref{thm:indfreerefl}
and Proposition \ref{prop:indfactoredindfree}.
See also \cite[Ex.\ 3.20]{hogeroehrle:factored}.

\begin{lemma}
\label{lem:grr4}
Let $W = G(r,r,\ell)$ for $r \ge 2$ and $\ell \ge 4$. 
Then $\CA(W)$ is not nice.
In particular, the Coxeter arrangement of $W(D_\ell) = G(2,2,\ell)$
for $\ell \ge 4$ is not nice.
\end{lemma}

\begin{proof}
Let $\ell \ge 4$ and $r \ge 2$.
It suffices to show the result for $G(r,r,4)$.
For, let $W = G(r,r,\ell)$ for $\ell \ge 5$.
Then noting that $G(r,r,4)$ is a parabolic subgroup of $W$,
the result follows for $W$ 
from Lemma \ref{lem:parabolicfactored}.

So let $r \ge 2$, 
$W = G(r,r,4)$ and $\CA = \CA(W)$.
We denote the coordinate functions of ${\mathbb C}^4$ by $x,y, z$ and $t$. 
Furthermore, let $\zeta$ be a primitive $r$-th root of $1$.
Out of six families of hyperplanes in $\CA$ we  
consider the following four,  given 
as follows. For $i = 0,\ldots, r-1$, set
\[
A_i := \ker(x-\zeta^i y), 
B_i := \ker(z-\zeta^i t),
C_i = \ker(x-\zeta^i z),
\text{and} \quad  
D_i := \ker(y-\zeta^i t). 
\] 
We also consider the hyperplanes
$A_i$, $B_i$, $C_i$ and $D_i$ for $i \in {\mathbb Z}$
simply by taking $i$ modulo $r$.

One readily checks that 
the following subarrangements $\CA_X$ of $\CA$ define 
members $X$ of rank $2$ of $L(\CA)$:
\[
\{A_0, A_1, \ldots, A_{r-1}\},
\{B_0, B_1, \ldots, B_{r-1}\},
\{C_0, C_1, \ldots, C_{r-1}\},
\{D_0, D_1, \ldots, D_{r-1}\},
\]
\[
\{A_i,  B_j \mid i,j \in \{0,\ldots, r-1\} \}
\text{ and, } 
\{C_i,  D_j \mid i,j \in \{0,\ldots, r-1\} \}.
\]

Suppose $\pi = (\pi_1, \pi_2, \pi_3, \pi_4)$ is a factorization of $\CA$.
The singleton condition 
from Definition \ref{def:factored}(ii) applied to 
$X = A_i \cap B_j$ for $i,j \in \{1, \ldots, r-1\}$ 
shows that the $A_i$ and $B_j$ are in distinct 
parts of $\pi$. 
Likewise, the singleton condition applied to 
$X = A_0 \cap A_1 \cap \ldots \cap A_{r-1}$ shows that not 
all the $A_i$ are in 
the same part of $\pi$. 
The same holds  for 
$B_0 \cap B_1 \cap \ldots \cap B_{r-1}$.
Thus, since 
$r \ge 2$, each part $\pi_i$ contains at least 
one of the 
hyperplanes $A_0, A_1, \ldots, A_{r-1}, B_0, B_1, \ldots, B_{r-1}$.

Applying the same argument to the hyperplanes
$C_0, C_1, \ldots, C_{r-1}, D_0, D_1, \ldots, D_{r-1}$ 
implies that each part of $\pi$ contains at least two hyperplanes.
This contradicts 
the fact that $\pi$ admits a part of cardinality $1$,
cf.\ Remark \ref{rem:factored}(iii). Thus
$\CA$ does not admit a nice partition.
\end{proof}

In \cite{terao:factored}, Terao already pointed out that 
the Coxeter arrangement of 
$W(D_4) = G(2,2,4)$ does not admit a factorization,
see Remark \ref{rem:d4} above.

\begin{lemma}
\label{lem:dn}
Let $W$ be of type $G_{33}$, $G_{34}$, $E_6$, $E_7$, or $E_8$. 
Then $\CA(W)$ is not nice.
\end{lemma}

\begin{proof}
Let  $W$ be as in the statement.
Then $W$ admits a parabolic subgroup of type $D_4$;
for $G_{33}$, $G_{34}$, see \cite[Table C.14 - C.23]{orlikterao:arrangements}.
By Lemma \ref{lem:grr4},
$\CA(W(D_4))$ is not nice.
It thus follows from Lemma \ref{lem:parabolicfactored}
that $\CA(W)$ is not nice either.
\end{proof}

\begin{lemma}
\label{lem:exceptional1}
Let $W$ be one of $H_3$, $G_{25}$, $H_4$, or $G_{32}$. 
Then $\CA(W)$ is not nice.
\end{lemma}

\begin{proof}
First we show that $\CA(H_3)$ and $\CA(G_{25})$ are not nice.
It then follows from  Lemma \ref{lem:parabolicfactored}
and \cite[Table C.13]{orlikterao:arrangements} that 
$\CA(H_4)$ and $\CA(G_{32})$ aren't nice either.

First let  $\CA = \CA(H_3)$ and let
$x, y$ and $z$ be the variables in $S$ 
and let $\zeta$ be a primitive $5$th root of unity.
We have 
\begin{align*}
Q(H_3) & = x y z (x + y) (y + z)  (x-(\zeta^3+\zeta^2)y) (x-(\zeta^3+\zeta^2+1)y)  \\
  & (x - (\zeta^3+\zeta^2)y-(\zeta^3+\zeta^2)z) (x - (\zeta^3+\zeta^2+1)y-(\zeta^3+\zeta^2+1)z) \\
 & (x + y + z) (x - (\zeta^3+\zeta^2)y-(\zeta^3+\zeta^2+1)z) (x - (\zeta^3+\zeta^2)y + z)  \\
 & (x + y + (\zeta^3+\zeta^2+2)z) (x + y - (\zeta^3+\zeta^2+1)z) \\
& (x - 2(\zeta^3+\zeta^2+1)y-(\zeta^3+\zeta^2+1)z).
\end{align*}
Set 
\begin{align*}
H_1 & = \ker x, H_2 = \ker y, H_3 = \ker z,
H_4 = \ker (x - (\zeta^3+\zeta^2)y-(\zeta^3+\zeta^2)z), \\
H_5 & = \ker (x + y + (\zeta^3+\zeta^2+2)z),\ \text{ and }
H_6 = \ker (x - 2(\zeta^3+\zeta^2+1)y-(\zeta^3+\zeta^2+1)z).
\end{align*}
One checks that 
each of the following intersections
describes a rank $2$ member $X \in L(\CA)$ with $\vert\CA_X\vert = 2$:
$H_1 \cap H_3$,  $H_1 \cap H_6$,  $H_3 \cap H_6$,  
$H_2 \cap H_4$,   $H_2 \cap H_5$,  and $H_4 \cap H_5$.
Suppose $\pi$ is a factorization of $\CA$.
Then applying the singleton condition of 
Definition \ref{def:factored}(ii) 
to each of these elements of $L(\CA)$
implies that 
each of $H_1, H_3, H_6$, respectively
each of $H_2, H_4, H_5$
must belong to a distinct part of $\pi$.
However, 
as one of the parts of $\pi$ has to have cardinality $1$,
cf.\ Remark \ref{rem:factored}(iii),
this is a contradiction.
Consequently, $\CA$ is not nice.

Next let  $\CA = \CA(G_{25})$ and again let
$x, y$ and $z$ be the variables in $S$ 
and let $\zeta$ be a primitive $3$rd root of unity.
We have 
\begin{align*}
Q(G_{25}) & = x y z (x + y + z) (x + y + \zeta z) (x + y - (\zeta+1) z)   \\
  & (x + \zeta y + z) (x + \zeta y + \zeta z) (x + \zeta y - (\zeta+1) z) \\
  & (x - (\zeta+1) y + z) (x - (\zeta+1) y + \zeta z) (x - (\zeta+1) y - (\zeta+1) z).
\end{align*}
Set 
$H_1  = \ker (x+y+z)$, $H_2 = \ker (x + y + \zeta z)$, 
$H_3 = \ker (x + \zeta y + z)$,
$H_4 = \ker (x + \zeta y - (\zeta+1) z)$,
$H_5  = \ker (x - (\zeta+1) y + \zeta z)$,  and 
$H_6 = \ker  (x - (\zeta+1) y - (\zeta+1) z)$.
One checks that 
each of the following intersections
describes a rank $2$ member $X \in L(\CA)$ with $\vert \CA_X \vert = 2$:
$H_1 \cap H_4$,  $H_1 \cap H_5$,  $H_4 \cap H_5$,  
$H_2 \cap H_3$,   $H_2 \cap H_6$,  and $H_3 \cap H_6$.
Suppose $\pi$ is a factorization of $\CA$.
Then, as above, applying the singleton condition of 
Definition \ref{def:factored}(ii) 
to each of these elements of $L(\CA)$ implies that 
each of $H_1, H_4, H_5$, respectively
each of $H_2, H_3, H_6$
must belong to a distinct part of $\pi$.
However, as one of the parts of $\pi$ has to have cardinality $1$,
cf.\ Remark \ref{rem:factored}(iii), 
this is a contradiction.
Consequently, $\CA$ is not nice.
\end{proof}

\begin{lemma}
\label{lem:g24}
Let $W = G_{24}$.
Then $\CA(W)$ is not nice.
\end{lemma}

\begin{proof}
Let  $\CA = \CA(G_{24})$ and let
$x, y$ and $z$ be the variables in $S$ 
and let $\zeta$ be a primitive $7$th root of unity.
We have 
\begin{align*}
Q(G_{24}) &= 
(x + (2 \zeta^4 + 2 \zeta^2 + 2 \zeta + 1) y)
(x + (-2 \zeta^4 - 2 \zeta^2 - 2 \zeta - 1) y) \\
& (3x + (-\zeta^4 - \zeta^2 - \zeta + 3) y -2(\zeta^4 + \zeta^2 + \zeta) z)\;  x\\
& (3x- (\zeta^4 + \zeta^2 + \zeta + 4) y -2(\zeta^4 + \zeta^2 + \zeta +1) z)\\
& (3x - 7 y + 4 z)\;  (3x + 7 y - 4 z)\\
& (3x + (2 \zeta^4 + 2 \zeta^2 + 2 \zeta + 1) y -2(\zeta^4 +\zeta^2 + \zeta -1) z)\\
& (7y + (-3 \zeta^4 - 3 \zeta^2 - 3 \zeta + 2) z)\\
& (3x + (-2 \zeta^4 - 2 \zeta^2 - 2 \zeta - 1) y -4(\zeta^4 + \zeta^2 + \zeta -1) z)\\
& (3x + (\zeta^4 + \zeta^2 + \zeta + 4) y + 2(\zeta^4 + \zeta^2 + \zeta + 1) z)\\
& (3x + (\zeta^4 + \zeta^2 + \zeta - 3) y + 2(\zeta^4 + \zeta^2 + \zeta) z)\\
& (7y + (3\zeta^4 + 3 \zeta^2 + 3 \zeta + 5) z)\\
& (3x + (2 \zeta^4 + 2 \zeta^2 + 2 \zeta + 1) y -2( \zeta^4 + \zeta^2 + \zeta +2) z)\\
& (3x + (-2 \zeta^4 - 2 \zeta^2 - 2 \zeta - 1) y + 2( \zeta^4 + \zeta^2 +  \zeta + 2) z)\\
& (7y + (-6 \zeta^4 - 6 \zeta^2 - 6 \zeta - 10) z)\\
& (3x + (-2 \zeta^4 - 2 \zeta^2 - 2 \zeta - 1) y + 2( \zeta^4 + \zeta^2 +  \zeta - 1) z)\\
& (3x + (-2 \zeta^4 - 2 \zeta^2 - 2\zeta - 1) y -4(\zeta^4 + \zeta^2 + \zeta +2) z)\\
& (3x + (2 \zeta^4 + 2 \zeta^2 + 2 \zeta + 1) y + 4( \zeta^4 + \zeta^2 + \zeta + 2) z)\\
& (3x + (2\zeta^4 + 2 \zeta^2 + 2 \zeta + 1) y + 4( \zeta^4 + \zeta^2 + \zeta - 1) z)\\
& (7y + (6 \zeta^4 + 6 \zeta^2 + 6 \zeta - 4) z).
\end{align*}
For simplicity, we enumerate the members of $\CA$ in the order they appear as factors in 
$Q(G_{24})$, i.e., 
$H_1 = \ker \big(x + (2 \zeta^4 + 2 \zeta^2 + 2 \zeta + 1) y\big)$, 
$H_2 = \ker \big(x + (-2 \zeta^4 - 2 \zeta^2 - 2 \zeta - 1) y\big)$, etc.
Suppose  $\CA$ admits a nice partition $\pi = \{\pi_1,\pi_2,\pi_3\}$. 
Since $W$ 
acts transitively on $\CA$, we may assume that $\pi_1 = \{H_1\}$.
The following four subsets $\CA_X$ of $\CA$ describe rank $2$ members $X$ of $L(\CA)$:
$\{H_1, H_{6}, H_{11}, H_{15}\}$, $\{H_1, H_7, H_{12}, H_{17}\}$, 
$\{H_1, H_8, H_{10}, H_{13}\}$, $\{H_1, H_9, H_{14}, H_{18}\}$, $\{H_1,H_2,H_4\}$, $
\{H_1,H_3,H_5\}$, $\{H_1,H_{16},H_{20}\}$ and $\{H_1, H_{19}, H_{21}\}$. 
It follows 
from Corollary \ref{cor:teraofactored}(iii) that each of 
$A:=\{H_{6}, H_{11}, H_{15}\}$, 
$B:=\{H_7, H_{12}, H_{17}\}$, $C:=\{H_8, H_{10}, H_{13}\}$, 
$D:=\{H_9, H_{14}, H_{18}\}$, 
$E:=\{H_2,H_4\}$,
$F:=\{H_3,H_5\}$, 
$G:= \{H_{16},H_{20}\}$ and 
$H:= \{H_{19}, H_{21}\}$ has to be contained in one part of $\pi$.

Clearly, as only $\pi_2$ and $\pi_3$ are candidates, 
at least two of the sets $A$, $B$,
$C$, $D$ have to be in one part. 
Applying the singleton condition 
in Definition \ref{def:factored}(ii) to each of  
$\CA_{X_1} = \{H_{6},H_{9},H_{10}\}$, $\CA_{X_2} = \{H_{6},H_{14},H_{17}\}$, 
$\CA_{X_3} = \{H_{7},H_{8},H_{15}\}$ and $\CA_{X_4} = \{H_{7},H_{13},H_{18}\}$, 
we conclude that 
\begin{eqnarray*}
A \cup B \in \pi_2 &\Leftrightarrow& C \cup D \in \pi_3,\\
A \cup C \in \pi_2 &\Leftrightarrow& B \cup D \in \pi_3, \text{ and}\\
A \cup D \in \pi_2 &\Leftrightarrow& B \cup C \in \pi_3.
\end{eqnarray*}
As the cardinality of the union of two sets in $\{A,B,C,D\}$ is even and 
as each of the 
remaining pairs of hyperplanes $\{E,F,G,H\}$ has to be added 
to either $\pi_2$ or $\pi_3$, it follows that $|\pi_2|$ is even.
But this is a contradiction, since $|\pi_2|\in\{9,11\}$ is odd 
(as $\exp \CA(G_{24}) = \{1,9,11\}$, cf.\ Remark \ref{rem:exponents}).
It follows that $\CA(W)$ is not nice, as claimed.
\end{proof}

\begin{lemma}
\label{lem:g26}
Let $W = G_{26}$.
Then $\CA(W)$ is not nice.
\end{lemma}

\begin{proof}
Let  $\CA = \CA(G_{26})$ and let
$x, y$ and $z$ be the variables in $S$ 
and let $\zeta$ be a primitive $3$rd root of unity.
We have 
\begin{align*}
  Q(G_{26}) &= 
  xyz (x - y) (x - \zeta y)
  (x - \zeta^2 y)
  (x - z)
  (x -\zeta z)
  (x - \zeta^2 z)\\
&(y - z)
  (y -\zeta z)
  (y - \zeta^2 z)
  (x + y + z)
  (x + \zeta y + z)
  (x + \zeta^2y + z)\\
&(x + y + \zeta z)
  (x + y + \zeta^2z)
  (x +\zeta^2 y + \zeta^2 z)
  (x + \zeta y + \zeta^2z)\\
&(x + \zeta^2 y + \zeta z)
  (x + \zeta y + \zeta z).
\end{align*}
Again, for simplicity, we enumerate the members of $\CA$ in the order they appear as factors in  $Q(G_{26})$, i.e., 
$H_1 = \ker x$,  $H_2 = \ker y$, etc.
Suppose  $\CA$ admits a nice partition $\pi = \{\pi_1,\pi_2,\pi_3\}$. 
There are two $W$-orbits in $\CA$, represented 
by $H_1$ and $H_4$. So we may assume that the singleton of $\pi$ 
is  either $\pi_1 = \{H_1\}$ or $\pi_1 = \{H_4\}$, respectively. 
Without loss, we may assume that 
$H_5 \in \pi_2$. 
Consider the rank $2$ elements $X_1,\ldots,X_4$ of $L(\CA)$ given by 
$\CA_{X_1} = \{H_5,H_3\}$, $\CA_{X_2} = \{H_5,H_{14}\}$, 
$\CA_{X_3} = \{H_5,H_{19}\}$ and $\CA_{X_4} = \{H_5,H_{21}\}$,
respectively.
It follows from 
Corollary \ref{cor:teraofactored}(iii) that 
$H_3,H_{14},H_{19},H_{21} \in \pi_3$. But this is a contradiction to 
Remark \ref{rem:factored}(iii), since 
$\CA_{H_3\cap H_{14} \cap H_{19} \cap H_{21}} = \{H_3,H_{14},H_{19},H_{21}\}$.
It follows that $\CA(W)$ is not nice, as desired.
\end{proof}

\begin{lemma}
\label{lem:g27}
Let $W = G_{27}$.
Then $\CA(W)$ is not nice.
\end{lemma}

\begin{proof}
Let  $\CA = \CA(G_{27})$ and let
$x, y$ and $z$ be the variables in $S$ 
and let $\zeta$ be a primitive $15$th root of unity.
We have 
\begin{align*}
Q(G_{27}) &= ( 15 x + \left(24 \zeta^{7} - 18 \zeta^{6} - 6 \zeta^{5} + 18 \zeta^{4} - 6 \zeta^{3} + 12 \zeta^{2} - 6 \zeta - 15\right) y \\
          & \qquad  + \left(-4 \zeta^{7} + 8 \zeta^{6} - 4 \zeta^{5} + 2 \zeta^{4} - 4 \zeta^{3} - 12 \zeta^{2} + 6 \zeta\right) z ) \\
&( 3 y + \left(\zeta^{7} + 3 \zeta^{5} - \zeta^{4} - \zeta^{3} + \zeta^{2} - \zeta + 2\right) z ) \\
&( 15 x + \left(-18 \zeta^{6} + 6 \zeta^{5} + 12 \zeta^{4} - 6 \zeta^{3} + 12 \zeta^{2} - 24 \zeta - 3\right) y \\
  & \qquad + \left(-12 \zeta^{6} + 4 \zeta^{5} + 8 \zeta^{4} - 4 \zeta^{3} + 8 \zeta^{2} - 16 \zeta - 2\right) z ) \\
&( 15 x + \left(-12 \zeta^{7} - 18 \zeta^{6} + 12 \zeta^{5} - 6 \zeta^{4} - 6 \zeta^{3} + 12 \zeta^{2} - 18 \zeta + 3\right) y \\
  & \qquad + \left(2 \zeta^{7} + 8 \zeta^{6} - 2 \zeta^{5} + 6 \zeta^{4} + 6 \zeta^{3} - 2 \zeta^{2} - 2 \zeta + 2\right) z ) \\
&( 15 x + \left(12 \zeta^{7} - 18 \zeta^{6} - 6 \zeta^{3} + 12 \zeta^{2} - 9\right) y \\
  &\qquad + \left(-2 \zeta^{7} + 28 \zeta^{6} - 10 \zeta^{5} + 6 \zeta^{3} - 22 \zeta^{2} + 20 \zeta + 4\right) z ) \\
&( 15 x + \left(9 \zeta^{7} - 12 \zeta^{6} - 3 \zeta^{5} + 9 \zeta^{4} - 9 \zeta^{3} + 3 \zeta^{2} - 3 \zeta - 9\right) y \\
  &\qquad + \left(\zeta^{7} - 8 \zeta^{6} + 3 \zeta^{5} - 4 \zeta^{4} - \zeta^{3} + 7 \zeta^{2} - 2 \zeta - 1\right) z ) \\
&( 15 x + \left(12 \zeta^{7} - 18 \zeta^{6} - 6 \zeta^{3} + 12 \zeta^{2} - 9\right) y \\
  &\qquad + \left(-22 \zeta^{7} + 8 \zeta^{6} + 10 \zeta^{5} + 6 \zeta^{3} - 2 \zeta^{2} - 20 \zeta + 14\right) z ) \\
&( 15 x + \left(6 \zeta^{7} + 6 \zeta^{6} + 12 \zeta^{3} + 6 \zeta^{2} + 3\right) y \\
  &\qquad + \left(14 \zeta^{7} - 16 \zeta^{6} + 10 \zeta^{4} - 2 \zeta^{3} + 14 \zeta^{2} - 10 \zeta - 8\right) z ) \\
&( 15 x + \left(-3 \zeta^{7} - 6 \zeta^{6} + 6 \zeta^{5} - 3 \zeta^{4} + 3 \zeta^{3} + 9 \zeta^{2} - 9 \zeta + 3\right) y \\
  & \qquad + \left(3 \zeta^{7} + 6 \zeta^{6} - \zeta^{5} - 2 \zeta^{4} + 7 \zeta^{3} + \zeta^{2} + 4 \zeta + 2\right) z ) \\
&( 15 x + \left(12 \zeta^{7} + 18 \zeta^{6} - 12 \zeta^{5} + 6 \zeta^{4} + 6 \zeta^{3} - 12 \zeta^{2} + 18 \zeta - 3\right) y \\
  & \qquad + \left(-2 \zeta^{7} - 8 \zeta^{6} + 2 \zeta^{5} - 6 \zeta^{4} - 6 \zeta^{3} + 2 \zeta^{2} + 2 \zeta - 2\right) z ) \\
&( 15 x + \left(30 \zeta^{7} - 6 \zeta^{6} - 18 \zeta^{5} + 24 \zeta^{4} - 12 \zeta^{3} - 6 \zeta^{2} + 12 \zeta - 21\right) y \\
  &\qquad + \left(20 \zeta^{7} - 4 \zeta^{6} - 12 \zeta^{5} + 16 \zeta^{4} - 8 \zeta^{3} - 4 \zeta^{2} + 8 \zeta - 14\right) z ) \\
&( 15 x + \left(6 \zeta^{7} + 6 \zeta^{6} + 12 \zeta^{3} + 6 \zeta^{2} + 3\right) y \\
  &\qquad + \left(-26 \zeta^{7} + 4 \zeta^{6} + 10 \zeta^{5} - 20 \zeta^{4} - 2 \zeta^{3} - 6 \zeta^{2} + 12\right) z ) \\
&( 3 y + \left(-2 \zeta^{7} - 2 \zeta^{5} - \zeta^{4} + 2 \zeta^{3} - 2 \zeta^{2} - \zeta + 1\right) z ) \\
&( 15 x + \left(3 \zeta^{7} + 6 \zeta^{6} - 6 \zeta^{5} + 3 \zeta^{4} - 3 \zeta^{3} - 9 \zeta^{2} + 9 \zeta - 3\right) y \\
  & \qquad + \left(-3 \zeta^{7} - 6 \zeta^{6} + \zeta^{5} + 2 \zeta^{4} - 7 \zeta^{3} - \zeta^{2} - 4 \zeta - 2\right) z ) \\
&( x + \left(2 \zeta^{6} + 2 \zeta^{3} + 1\right) y ) \\
&( 15 x + \left(-12 \zeta^{7} + 6 \zeta^{5} - 3 \zeta^{4} - 9 \zeta + 6\right) y \\
  & \qquad + \left(-3 \zeta^{7} - \zeta^{5} - 2 \zeta^{4} - 5 \zeta^{3} - 5 \zeta^{2} + 4 \zeta - 1\right) z ) \\
&( 15 x + \left(-6 \zeta^{7} + 6 \zeta^{6} + 6 \zeta^{5} - 18 \zeta^{4} + 12 \zeta^{3} + 6 \zeta^{2} + 6 \zeta + 9\right) y \\
  &\qquad + \left(-14 \zeta^{7} + 4 \zeta^{6} + 4 \zeta^{5} - 2 \zeta^{4} - 2 \zeta^{3} - 6 \zeta^{2} - 6 \zeta + 6\right) z ) \\
&( 15 x + \left(-12 \zeta^{7} + 18 \zeta^{6} + 6 \zeta^{3} - 12 \zeta^{2} + 9\right) y \\
  &\qquad + \left(22 \zeta^{7} - 8 \zeta^{6} - 10 \zeta^{5} - 6 \zeta^{3} + 2 \zeta^{2} + 20 \zeta - 14\right) z ) \\
&( 15 x + \left(-6 \zeta^{7} - 6 \zeta^{6} - 12 \zeta^{3} - 6 \zeta^{2} - 3\right) y \\
  &\qquad + \left(-14 \zeta^{7} + 16 \zeta^{6} - 10 \zeta^{4} + 2 \zeta^{3} - 14 \zeta^{2} + 10 \zeta + 8\right) z ) \\
&( x + \left(2 \zeta^{7} - 2 \zeta^{6} + 2 \zeta^{2} - 1\right) y ) \
 ( x + \left(-2 \zeta^{6} - 2 \zeta^{3} - 1\right) y ) \\
&( 3 y + \left(-\zeta^{7} + 2 \zeta^{5} + \zeta^{4} + \zeta^{3} - \zeta^{2} + \zeta + 2\right) z ) \\
&( 15 x + \left(6 \zeta^{7} + 6 \zeta^{6} + 12 \zeta^{3} + 6 \zeta^{2} + 3\right) y \\
  &\qquad + \left(14 \zeta^{7} - 16 \zeta^{6} - 10 \zeta^{5} + 20 \zeta^{4} - 22 \zeta^{3} - 6 \zeta^{2} - 18\right) z ) \\
&( 15 x + \left(-6 \zeta^{7} + 3 \zeta^{5} - 9 \zeta^{4} + 3 \zeta + 3\right) y \\
  &\qquad + \left(11 \zeta^{7} - 10 \zeta^{6} - 3 \zeta^{5} + 4 \zeta^{4} - 5 \zeta^{3} + 5 \zeta^{2} + 2 \zeta - 8\right) z ) \\
&( 15 x + \left(18 \zeta^{6} - 6 \zeta^{5} - 12 \zeta^{4} + 6 \zeta^{3} - 12 \zeta^{2} + 24 \zeta + 3\right) y \\
  &\qquad + \left(12 \zeta^{6} - 4 \zeta^{5} - 8 \zeta^{4} + 4 \zeta^{3} - 8 \zeta^{2} + 16 \zeta + 2\right) z ) \\
&( 15 x + \left(-6 \zeta^{7} - 6 \zeta^{6} - 12 \zeta^{3} - 6 \zeta^{2} - 3\right) y \\
  & \qquad + \left(26 \zeta^{7} - 4 \zeta^{6} - 10 \zeta^{5} + 20 \zeta^{4} + 2 \zeta^{3} + 6 \zeta^{2} - 12\right) z ) \\
&( 15 x + \left(-12 \zeta^{7} + 18 \zeta^{6} + 6 \zeta^{3} - 12 \zeta^{2} + 9\right) y \\
  &\qquad + \left(-18 \zeta^{7} + 12 \zeta^{6} + 10 \zeta^{5} - 10 \zeta^{4} + 14 \zeta^{3} + 2 \zeta^{2} - 10 \zeta + 16\right) z ) \\
&( 15 x + \left(-9 \zeta^{7} + 12 \zeta^{6} + 3 \zeta^{5} - 9 \zeta^{4} + 9 \zeta^{3} - 3 \zeta^{2} + 3 \zeta + 9\right) y \\
  &\qquad + \left(-\zeta^{7} + 8 \zeta^{6} - 3 \zeta^{5} + 4 \zeta^{4} + \zeta^{3} - 7 \zeta^{2} + 2 \zeta + 1\right) z )  \\
& x\  ( x + \left(-2 \zeta^{7} + 2 \zeta^{6} - 2 \zeta^{2} + 1\right) y ) \\
&( 15 x + \left(-24 \zeta^{7} + 18 \zeta^{6} + 6 \zeta^{5} - 18 \zeta^{4} + 6 \zeta^{3} - 12 \zeta^{2} + 6 \zeta + 15\right) y \\
  &\qquad + \left(4 \zeta^{7} - 8 \zeta^{6} + 4 \zeta^{5} - 2 \zeta^{4} + 4 \zeta^{3} + 12 \zeta^{2} - 6 \zeta\right) z ) \\
&( 15 x + \left(-18 \zeta^{7} + 6 \zeta^{6} + 12 \zeta^{5} - 6 \zeta^{4} + 12 \zeta^{3} + 6 \zeta^{2} - 18 \zeta + 15\right) y \\
  &\qquad + \left(8 \zeta^{7} - 16 \zeta^{6} - 2 \zeta^{5} + 6 \zeta^{4} - 12 \zeta^{3} + 4 \zeta^{2} - 2 \zeta - 10\right) z ) \\
&( 15 x + \left(6 \zeta^{7} - 3 \zeta^{5} + 9 \zeta^{4} - 3 \zeta - 3\right) y \\
  &\qquad + \left(-11 \zeta^{7} + 10 \zeta^{6} + 3 \zeta^{5} - 4 \zeta^{4} + 5 \zeta^{3} - 5 \zeta^{2} - 2 \zeta + 8\right) z ) \\
&( 15 x + \left(18 \zeta^{7} - 6 \zeta^{6} - 12 \zeta^{5} + 6 \zeta^{4} - 12 \zeta^{3} - 6 \zeta^{2} + 18 \zeta - 15\right) y \\
  &\qquad + \left(-8 \zeta^{7} + 16 \zeta^{6} + 2 \zeta^{5} - 6 \zeta^{4} + 12 \zeta^{3} - 4 \zeta^{2} + 2 \zeta + 10\right) z ) \\
&( 15 x + \left(-6 \zeta^{7} - 6 \zeta^{6} - 12 \zeta^{3} - 6 \zeta^{2} - 3\right) y \\
  &\qquad + \left(-4 \zeta^{7} - 4 \zeta^{6} + 10 \zeta^{4} - 8 \zeta^{3} - 4 \zeta^{2} - 10 \zeta - 2\right) z ) \\
&( 15 x + \left(-12 \zeta^{7} + 18 \zeta^{6} + 6 \zeta^{3} - 12 \zeta^{2} + 9\right) y \\
  &\qquad + \left(12 \zeta^{7} + 12 \zeta^{6} - 10 \zeta^{5} + 10 \zeta^{4} + 4 \zeta^{3} - 8 \zeta^{2} + 10 \zeta - 4\right) z ) \\
&( 15 x + \left(-12 \zeta^{7} + 18 \zeta^{6} + 6 \zeta^{3} - 12 \zeta^{2} + 9\right) y \\
  &\qquad + \left(2 \zeta^{7} - 28 \zeta^{6} + 10 \zeta^{5} - 6 \zeta^{3} + 22 \zeta^{2} - 20 \zeta - 4\right) z ) \\
&( 15 x + \left(-30 \zeta^{7} + 6 \zeta^{6} + 18 \zeta^{5} - 24 \zeta^{4} + 12 \zeta^{3} + 6 \zeta^{2} - 12 \zeta + 21\right) y \\
  &\qquad + \left(-20 \zeta^{7} + 4 \zeta^{6} + 12 \zeta^{5} - 16 \zeta^{4} + 8 \zeta^{3} + 4 \zeta^{2} - 8 \zeta + 14\right) z ) \\
&( 15 x + \left(6 \zeta^{7} + 6 \zeta^{6} + 12 \zeta^{3} + 6 \zeta^{2} + 3\right) y \\
  &\qquad + \left(4 \zeta^{7} + 4 \zeta^{6} - 10 \zeta^{4} + 8 \zeta^{3} + 4 \zeta^{2} + 10 \zeta + 2\right) z ) \\
&( 15 x + \left(12 \zeta^{7} - 18 \zeta^{6} - 6 \zeta^{3} + 12 \zeta^{2} - 9\right) y \\
  &\qquad + \left(-12 \zeta^{7} - 12 \zeta^{6} + 10 \zeta^{5} - 10 \zeta^{4} - 4 \zeta^{3} + 8 \zeta^{2} - 10 \zeta + 4\right) z ) \\
&( 15 x + \left(-6 \zeta^{7} - 6 \zeta^{6} - 12 \zeta^{3} - 6 \zeta^{2} - 3\right) y \\
  &\qquad + \left(-14 \zeta^{7} + 16 \zeta^{6} + 10 \zeta^{5} - 20 \zeta^{4} + 22 \zeta^{3} + 6 \zeta^{2} + 18\right) z ) \\
&( 15 x + \left(12 \zeta^{7} - 6 \zeta^{5} + 3 \zeta^{4} + 9 \zeta - 6\right) y \\
  &\qquad + \left(3 \zeta^{7} + \zeta^{5} + 2 \zeta^{4} + 5 \zeta^{3} + 5 \zeta^{2} - 4 \zeta + 1\right) z ) \\
&( 15 x + \left(6 \zeta^{7} - 6 \zeta^{6} - 6 \zeta^{5} + 18 \zeta^{4} - 12 \zeta^{3} - 6 \zeta^{2} - 6 \zeta - 9\right) y \\
  & \qquad + \left(14 \zeta^{7} - 4 \zeta^{6} - 4 \zeta^{5} + 2 \zeta^{4} + 2 \zeta^{3} + 6 \zeta^{2} + 6 \zeta - 6\right) z ) \\
&( 15 x + \left(12 \zeta^{7} - 18 \zeta^{6} - 6 \zeta^{3} + 12 \zeta^{2} - 9\right) y \\
  &\qquad + \left(18 \zeta^{7} - 12 \zeta^{6} - 10 \zeta^{5} + 10 \zeta^{4} - 14 \zeta^{3} - 2 \zeta^{2} + 10 \zeta - 16\right) z ) \\
&( 3 y + \left(2 \zeta^{7} - 3 \zeta^{5} + \zeta^{4} - 2 \zeta^{3} + 2 \zeta^{2} + \zeta - 2\right) z ).
\end{align*}
For simplicity, once again 
we enumerate the members of $\CA$ in the order they appear as factors in 
$Q(G_{27})$.
Suppose  $\CA$ admits a nice partition $\pi = \{\pi_1,\pi_2,\pi_3\}$. 
Since $W$ acts transitively on $\CA(W)$, we can assume without loss
that $\pi_1 = \{H_1\}$. 
We derive a contradiction by considering the following $25$ 
rank $2$ members $X_i$ of $L(\CA)$:
\[
\begin{array}{ll}
\CA_{X_1} = \{H_1,H_2,H_3,H_4,H_5\}, & \CA_{X_2} = \{H_1,H_6,H_7,H_8,H_9\}, \\
\CA_{X_3} = \{H_1,H_{10},H_{11},H_{12},H_{13}\}, & \CA_{X_4} = \{H_1,H_{14},H_{15},H_{16},H_{17}\}, \\ 
 \CA_{X_5} = \{H_{4}, H_{6}, H_{18}\}, & \CA_{X_6} = \{H_{3}, H_{9}, H_{19}\},\\
\CA_{X_7} = \{H_{2},H_{7},H_{20}\}, &  \CA_{X_8} = \{H_{5},  H_{8},H_{21} \}, \\ 
\CA_{X_9} = \{H_{3}, H_{7},  H_{22}\},  &\CA_{X_{10}} = \{H_{5},H_{6},H_{23}\},\\ 
 \CA_{X_{11}} = \{H_{19}, H_{23}, H_{24}\}, & \CA_{X_{12}} = \{H_{18}, H_{22}, H_{25} \},\\
\CA_{X_{13}} = \{H_{21},H_{22},H_{26}\}, &  \CA_{X_{14}} = \{H_{20}, H_{23}, H_{27}\}, \\
 \CA_{X_{15}} = \{H_{24}, H_{25}, H_{26}, H_{27}\}, & \CA_{X_{16}} = \{H_{11},H_{17},H_{18}\}, \\
 \CA_{X_{17}} = \{ H_{11}, H_{16}, H_{20}\}, & \CA_{X_{18}} = \{ H_{10}, H_{16},H_{21}\},\\
\CA_{X_{19}} = \{H_{10},H_{17},H_{19}\}, &  \CA_{X_{20}} = \{H_{12}, H_{15}, H_{22} \}, \\
 \CA_{X_{21}} = \{H_{13}, H_{15}, H_{23} \}, &\CA_{X_{22}} = \{H_{4},H_{12},H_{25}\}, \\ 
 \CA_{X_{23}} = \{H_{8}, H_{14}, H_{25} \}, & \CA_{X_{24}} = \{H_{3},H_{15},H_{24}\}, \\ 
 \CA_{X_{25}} = \{H_{9}, H_{13}, H_{24} \}.\\
\end{array}
\]
Applying the singleton condition from 
Definition \ref{def:factored}(ii) to 
$\CA_{X_1},\CA_{X_2},\CA_{X_3}$, and $\CA_{X_4}$ shows that 
each of the sets $A:= \{H_2,H_3,H_4,H_5\}$, $B:= \{H_6,H_7,H_8,H_9\}$, 
$C:=\{H_{10},H_{11},H_{12},H_{13}\}$ and $D:=\{H_{14},H_{15},H_{16},H_{17}\}$
is a subset of $\pi_2$ or $\pi_3$. 
First suppose that $A \cup B \subset \pi_2$. Then
the singleton condition Definition \ref{def:factored}(ii)
applied to each of 
$\CA_{X_5},\CA_{X_6},\CA_{X_7},\CA_{X_8},\CA_{X_9}$, and $\CA_{X_{10}}$ 
shows that $H_{18}, H_{19},H_{20},H_{21}, H_{22}, H_{23} \in \pi_3$. The 
same argument 
applied to  $\CA_{X_{11}}, \CA_{X_{12}},\CA_{X_{13}}$, and $\CA_{X_{14}}$
shows that $H_{24}, H_{25}, H_{26}, H_{27} \in \pi_2$. This now leads to a 
contradiction, since $\CA_{X_{15}}$ then violates the singleton condition.
Now suppose that $C \cup D \subset \pi_2$. Then 
Remark \ref{rem:factored}(iii) applied to 
$\CA_{X_{16}},\CA_{X_{17}},\CA_{X_{18}},\CA_{X_{19}},\CA_{X_{20}}$,
and $\CA_{X_{21}}$ 
shows that $H_{18}, H_{19},H_{20},H_{21}, H_{22}, H_{23} \in \pi_3$. This 
however 
leads again to the same contradiction as in the case $A \cup B \subset \pi_2$. 
It follows that exactly two of the sets $A,B,C,D$ have to be 
contained in $\pi_2$.
Now suppose that $A \cup C \subset \pi_2$ and $B \cup D \subset \pi_3$. The singleton
condition applied to $\CA_{X_{22}}$ and $\CA_{X_{23}}$ implies that 
$H_{25} \in \pi_2 \cap \pi_3$, which is absurd.
Finally, suppose that $A \cup D \subset \pi_2$ and 
$B \cup C \subset \pi_3$. The singleton
condition applied to $\CA_{X_{24}}$ and $\CA_{X_{25}}$ implies that 
$H_{24} \in \pi_2 \cap \pi_3$, which is absurd.
It follows that $\CA(G_{27})$ is not nice.
\end{proof}

\begin{lemma}
\label{lem:f4}
Let $W$ be of type $F_4$.
Then $\CA(W)$ is not nice.
\end{lemma}

\begin{proof}
Let  $\CA = \CA(F_4)$ and let
$u,x, y$ and $z$ be the variables in $S$.
We have 
\begin{align*}
Q(F_4) & = u x y z (u + x) (x + y) (y + z) (u + x + y) (x + 2 y) (x + y + z)\\
& (u + x + 2 y) (u + x + y + z) (x + 2 y + z) (u + 2 x + 2 y) (u + x + 2 y + z)\\
& (x + 2 y + 2 z) (u + 2 x + 2 y + z) (u + x + 2 y + 2 z) (u + 2 x + 3 y + z) \\
& (u + 2 x + 2 y + 2 z) (u + 2 x + 3 y + 2 z) (u + 2 x + 4 y + 2 z) \\
& (u + 3 x + 4 y + 2 z) (2u + 3 x + 4 y + 2z).
\end{align*}
For simplicity, we enumerate the members of $\CA$ in the order they appear as factors in 
$Q(F_4)$, i.e., 
$H_1 = \ker u$, 
$H_2 = \ker x$, etc.
Suppose that $\pi = (\pi_1,\pi_2,\pi_3,\pi_4)$ is a nice 
partition of $\CA$. The following argument holds for every 
choice of a singleton $\pi_1 = \{H_i\}$. Without loss, let $\pi_1 = \{H_1\}$.
Consider the following rank $2$ members of $L(\CA)$:
\begin{align*}
\CA_{X_1} &= \{H_1,H_2,H_5\},  \CA_{X_2} = \{H_1,H_6,H_8,H_{14}\}, \\
\CA_{X_3} &= \{H_1,H_{10},H_{12},H_{20}\}, \text{ and }
\CA_{X_4} = \{H_1,H_{13},H_{15},H_{22}\}.
\end{align*}
The singleton condition from 
Definition \ref{def:factored}(ii)
applied to each of
$X_1,X_2,X_3$, and $X_4$ 
shows that each of the sets $A:= \{H_2,H_5\}$, $B:= \{H_6,H_8,H_{14}\}$,
$C:=\{H_{10},H_{12},H_{20}\}$, and $D:=\{H_{13},H_{15},H_{22}\}$ is a subset
of $\pi_2, \pi_3$ or $\pi_4$. Now, applying the singleton condition to 
\begin{align*}
 \CA_{X_5} &= \{H_2,H_8\}, \CA_{X_6} = \{H_2,H_{12}\}, \CA_{X_7} = \{H_2,H_{13}\},\\
 \CA_{X_8} &= \{ H_6,H_{20} \}, \CA_{X_9} = \{ H_6,H_{22} \}, \text{ and } 
 \CA_{X_{10}} = \{H_{10},H_{22}\}
\end{align*}
 shows that $A,B,C$ and $D$ have to be in different blocks of the factorization.
This contradicts the fact that $\pi_1 = \{H_1\}$.
It follows that $\CA$ is not nice.
\end{proof}

\begin{lemma}
\label{lem:g29}
Let $W = G_{29}$.
Then $\CA(W)$ is not nice. 
\end{lemma}

\begin{proof}
Let  $\CA = \CA(G_{29})$ and let
$u, x, y$ and $z$ be the variables in $S$. 
We have 
\begin{align*}
  Q(G_{29}) & = z (u - x + iy + i z) (u - x) (x - y) (u - x + i y - i z) (y + z)
             (u + i x - y + i z)\\
         &   (u - x - i y - i z) (u - y) (u - x - i y + i z) (u - i x + i y + z)
             (u + i x - y - i z)\\
         &   (x + z) (u - i x - y - i z) (u - i x - y + i z) (u + i x - i y + z)
             (y - z) (u - i x + i y - z)\\
         &   (x - z) (u + i x - i y - z) y (u + z) (u + i x + y + i z)
             (u + i x + i y + z) (u + i x + i y - z)\\
         &   (u - z) (u - i x - i y + z) (u - i x + y - i z) (u - i x + y + i z)
              x (u + x + i y + i z) \\
         &   (u + x - i y - i z) (u + x - i y + i z) (u + i x + y - i z) 
             (u - i x - i y - z) \\
         &   (u + x + i y - i z) (x + y) u (u + y) (u + x).
\end{align*}
Again, for simplicity, we enumerate the members of $\CA$ in the order they appear as factors in  $Q(G_{29})$, i.e., 
$H_1 = \ker z$,  $H_2 =\ker (u - x + iy + i z)$, etc.
Suppose that $\pi = (\pi_1,\pi_2,\pi_3,\pi_4)$ is a nice 
partition of $\CA$. 
We may assume that $\pi_1 = \{H_1\}$ is the singleton, since $W$ is
transitive on $\CA$.  Corollary \ref{cor:teraofactored}(iii) applied to 
$\CA_{H_1 \cap H_6} = \{H_1,H_6,H_{17},H_{21}\}$, 
$\CA_{H_1 \cap H_{22}} = \{H_1,H_{22},H_{26},H_{38}\}$ and 
$\CA_{H_1 \cap H_{13}} = \{H_1,H_{13},H_{19},H_{30}\}$ shows that each of the 
sets $A:= \{H_6,H_{17},H_{21}\}$, $B:= \{H_{22},H_{26},H_{38}\}$ and
$C:= \{H_{13},H_{19},H_{30}\}$ is contained in one of the parts of $\pi$. 
Moreover, these sets are in different parts of $\pi$, thanks to 
Remark \ref{rem:factored}(iii) applied to 
$\CA_{H_6 \cap H_{38}} = \{H_6, H_{38}\}$, 
$\CA_{H_6 \cap H_{30}} = \{H_6, H_{30}\}$ and
$\CA_{H_{13} \cap H_{38}} = \{H_{13}, H_{38}\}$. 
So lets assume that $A \subset \pi_2$,
$B \subset \pi_3$ and $C \subset \pi_4$. 
Consider $\CA_X = \{H,H'\}$, where $H \in A$ and $H'$ is one of the following 
hyperplanes
$H_2$, $H_3$, $H_5$, $H_8$, $H_{10}$, $H_{13}$, 
$H_{19}$, $H_{22}$, $H_{26}$, $H_{30}$, $H_{31}$, $H_{32}$, $H_{33}$, 
$H_{36}$, $H_{38}$ and $H_{40}$. 
Thanks to Definition \ref{def:factored}(ii), it follows that none of 
the $H'$ is contained in $\pi_2$. 
Using the same 
argument, 
considering $\CA_X = \{H,H'\}$, where this time
$H \in C$ and $H'$ is one of the following 
hyperplanes 
$H_{6}$, $H_{7}$, $H_{9}$, $H_{12}$, $H_{14}$, $H_{15}$, $H_{17}$, 
$H_{21}$, $H_{22}$, $H_{23}$, $H_{26}$, $H_{28}$, $H_{29}$, $H_{34}$, 
$H_{38}$ and $H_{39}$, it follows 
again from Definition \ref{def:factored}(ii) that none of the $H'$ is 
contained in $\pi_4$.

Next we consider the following rank $3$ members $Y_i$ of $L(\CA)$:
\begin{align*}
\CA_{Y_1} & = \{H_1, H_2, H_3, H_5, H_6, H_8, H_{10}, H_{17}, H_{21}\}, \\
\CA_{Y_2} & = \{H_1, H_6, H_{17}, H_{21}, H_{31}, H_{32}, H_{33}, H_{36}, H_{40}\}, \\
\CA_{Y_3} & = \{H_1, H_7, H_9, H_{12}, H_{13}, H_{14}, H_{15}, H_{19}, H_{30}\},
\text{ and } \\
\CA_{Y_4} & = \{H_1, H_{13}, H_{19}, H_{23}, H_{28}, H_{29}, H_{30}, H_{34}, H_{39}\}.
\end{align*}

Now using the information above, it follows from 
Corollary \ref{cor:teraofactored}(iii) that each of the following four sets 
is contained in one of the parts of $\pi$. 
\begin{align*}
Z_1 & := \{H_2,H_3,H_5,H_8,H_{10} \}, 
& Z_2  := \{H_{31},H_{32},H_{33},H_{36},H_{40} \}, \\ 
Z_3 & := \{H_7,H_9,H_{12},H_{14},H_{15}\}, \text{ and }
& Z_4 := \{H_{23},H_{28},H_{29},H_{34},H_{39}\}.
\end{align*}

The singleton condition 
from Definition \ref{def:factored}(ii) applied to 
each of $\{H_2,H_{12}\}$, $\{H_2,H_{28}\}$,
$\{H_3,H_{31}\}$, $\{H_7,H_{32}\}$, $\{H_7,H_{39}\}$ and $\{H_{23},H_{36}\}$ 
shows that each of the $Y_i$'s lies in a different part of $\pi$.
But this contradicts the fact that $\pi_1 = \{H_1\}$.
It follows that $\CA$ is not nice.
\end{proof}

\begin{lemma}
\label{lem:g31}
Let $W = G_{31}$.
Then $\CA(W)$ is not nice.
\end{lemma}

\begin{proof}
Let  $\CA = \CA(G_{31})$ and let
$u, x, y$, and $z$ be the variables in $S$. 
We have 
\begin{align*}
Q(G_{31}) & = u (u + i x) (u - x) (u + x + y + z) (x - y) (u -i x) (u + x) 
          (u - x - y - z) x \\
       & (u - x -i y -i z) (u + iy) (u + x -iy -iz) (u - y) (u - x + y + z)
         ( u - x + iy + iz) \\ 
       & (x + iy) (y + z) (u -ix - y -iz) (u + x + iy + iz) (u + ix -iy + z)
         (u -iy) (u + y)\\
       &  y (u -i x + y -iz) (u + x - y + z) (u + ix - y + iz) (x -i y)
         (x + z) (u + ix + y + iz) \\
       & (u -ix + iy + z) (u + x - y - z) (u -ix + iy - z) (u - x + y - z)
         (u +ix -iy - z) \\
       & (u + ix + iy - z) (x + y) (u - x + iy -z) (u + ix + y -iz) (u -iz) 
         (u + ix - y -iz)\\
       & (u + x + iy -iz) (u + ix + iy + z) (u + z) (u -ix -iy + z)
         (u - x - y + z) \\
       & (u -ix - y + iz) (u -ix + y + iz) (u - x -iy + iz) (u -ix -iy - z) 
         (x -iz) (x + iz) \\
       & (y + iz) (u + iz) (u + x -iy + iz)(y -iz)(u - z) (u + x + y - z) 
         (x - z) z (y - z).
\end{align*}
Again, for simplicity, we enumerate the members of $\CA$ in the order they appear as factors in  $Q(G_{31})$, i.e., 
$H_1 = \ker u$,  $H_2 =\ker (u + i x)$, etc.
Suppose that $\pi = (\pi_1,\pi_2,\pi_3,\pi_4)$ is a nice 
partition of $\CA$. 
Since $G_{31}$ acts transitively on $\CA$, we may 
assume that $\pi_{1} = \{H_1\}$. Then Corollary \ref{cor:teraofactored}(iii) applied to 
$\CA_{H_1 \cap H_2} = \{ H_1,H_2,H_3,H_6,H_7,H_9\}$, 
$\CA_{H_1 \cap H_{11}} = \{ H_1, H_{11}, H_{13}, H_{21}, H_{22}, H_{23}\}$
and $\CA_{H_1 \cap H_{39}} = \{H_1, H_{39}, H_{43}, H_{53}, H_{56}, H_{59}\}$ 
shows that each of $A:= \{H_2,H_3,H_6,H_7,H_9\}$, 
$B := \{H_{11}, H_{13}, H_{21}, H_{22}, H_{23}\}$ and 
$C:= \{H_{39}, H_{43}, H_{53}, H_{56}, H_{59}\}$ is a subset of one of the
parts of $\pi$. Moreover, $A,B,C$ have to be in distinct parts of $\pi$,
thanks to Remark \ref{rem:factored}(iii) applied to 
$\CA_{H_2 \cap H_{23}} =\{H_2,H_{23}\}$, 
$\CA_{H_2 \cap H_{59}} = \{H_2,H_{59}\}$ and 
$\CA_{H_{11}\cap H_{59}} = \{H_{11}, H_{59}\}$.
But now applying 
the singleton condition from 
Remark \ref{rem:factored}(iii) to  
$\CA_{H_3\cap H_4} = \{H_3, H_4\}$,
$\CA_{H_4\cap H_{13}} = \{H_4,H_{13}\}$ and 
$\CA_{H_4\cap H_{56}} = \{H_4,H_{56}\}$ shows that 
$H_4 \notin \pi_1\cup\pi_2\cup \pi_3 \cup \pi_4 = \CA$
which is absurd.
We conclude that $\CA$ is not nice.
\end{proof}

\subsection{Proof of Theorem \ref{thm:heredfactored}.}
\label{ss:heredfactored}

Thanks to Corollary \ref{cor:product-heredindfactored},
the question of the presence of a 
hereditary (inductive) factorization reduces 
to the case when $\CA$ is irreducible. 
Thus we may assume that $W$ is irreducible.

The reverse implication of Theorem \ref{thm:heredfactored} is clear.
So suppose that $W$ is irreducible so that 
$\CA = \CA(W)$ is nice.
We need to show that 
$\CA^X$ is also nice for every $X \in L(\CA)$.
If $\CA$ is supersolvable, then so is 
$\CA^X$ for every $X \in L(\CA)$, 
by \cite[Prop.\ 3.2]{stanley:super}.
Consequently, $\CA^X$ is factored again, 
thanks to Proposition \ref{prop:ssfactored}.
For $W$ of rank $3$, the result follows 
from Lemma \ref{lem:3-arr}.
Thus, 
Theorem \ref{thm:heredfactored} follows from 
Theorem \ref{thm:factoredrefl}.

\bigskip
Finally, we comment on questions of
computations underlying this work.

\begin{remark}
\label{rem:computations}
In order to establish several of our results 
we first use the functionality for complex reflection groups 
provided by the   \CHEVIE\ package in   \GAP\ 
(and some \GAP\ code by J.~Michel \cite{michel:development})
(see \cite{gap3} and \cite{chevie})
in order to obtain explicit 
linear functionals $\alpha$ defining the hyperplanes 
$H = \ker \alpha$ of the reflection arrangement
$\CA(W)$. 
We then use the functionality of
\Sage\  (\cite{sage}) to explicitly 
determine the intersection lattice
$L(\CA(W))$ in the relevant instances,
see Lemmas \ref{lem:g24} - \ref{lem:g31}.
\end{remark}





\bigskip

\bibliographystyle{amsalpha}

\begin{thebibliography}{GHL{\etalchar{+}}96}

\bibitem[AHR14]{amendhogeroehrle:indfree}
N.~Amend, T.~Hoge and G.~R\"ohrle, 
\emph{On inductively free restrictions of reflection arrangements},
 J. Algebra \textbf{418} (2014), 197--212.

\bibitem[BC12]{cuntz:indfree} M. Barakat and M. Cuntz, \emph{Coxeter and 
    crystallographic arrangements are inductively free}, Adv. Math \textbf{229}
    (2012), 691--709.

\bibitem[Bou68]{bourbaki:groupes} N.~Bourbaki, \emph{\'{E}l\'ements de
    math\'ematique. {G}roupes et alg\`ebres de {L}ie. {C}hapitre {IV}-{VI}},
  Actualit\'es Scientifiques et Industrielles, No. 1337, Hermann, Paris,
  1968.

\bibitem[GHL{\etalchar{+}}96]{chevie} M.~Geck, G.~Hi{\ss}, F.~L{\"u}beck,
  G.~Malle, and G.~Pfeiffer, \emph{{CHEVIE} --- {A} system for computing and
    processing generic character tables}, Appl.  Algebra
  Engrg. Comm. Comput. \textbf{7} (1996), 175--210.

\bibitem[HR14]{hogeroehrle:super} 
T.~Hoge and G.~R\"ohrle, 
\emph{On supersolvable reflection arrangements}, 
 Proc. AMS, \textbf{142} (2014), no. 11, 3787--3799.

\bibitem[HR15a]{hogeroehrle:indfree} 
\bysame, 
\emph{On inductively free reflection arrangements}, 
J.~Reine u.~Angew.~Math. 
\textbf{701} (2015), 205--220.

\bibitem[HR15b]{hogeroehrle:factored} 
\bysame, 
\emph{Addition-Deletion Theorems for 
Factorizations of Orlik-Solomon Algebras and
nice Arrangements},
\url{http://arXiv.org/abs/1402.3227}.

\bibitem[J90]{jambu:factored}
M.~Jambu,
\emph{Fiber-type arrangements and factorization properties}.
Adv. Math. \textbf{80} (1990), no. 1, 1--21. 

\bibitem[JP95]{jambuparis:factored}
M.~Jambu and L.~Paris,
\emph{Combinatorics of Inductively Factored Arrangements},
Europ. J. Combinatorics \textbf{16} (1995), 267--292.

\bibitem[JT84]{jambuterao:free} 
M.~Jambu and H.~Terao, 
\emph{Free arrangements of hyperplanes and 
supersolvable lattices},
Adv. in Math. \textbf{52} (1984), no.~3, 248--258. 

\bibitem[M15]{michel:development}
J.~Michel.
\newblock{\em The development version of the CHEVIE package of GAP3}.
J. Algebra, \textbf{435} (2015), 308--336.

\bibitem[OS80]{orliksolomon:hyperplanes} 
P.~Orlik and L.~Solomon, 
\emph{Combinatorics and topology of complements of hyperplanes}, 
Invent. math. \textbf{56} (1980), 77--94.

\bibitem[OS82]{orliksolomon:unitaryreflectiongroups}
\bysame, 
\emph{Arrangements Defined by Unitary Reflection Groups},
Math. Ann. \textbf{261}, (1982), 339--357.

\bibitem[OST84]{orliksolomonterao:hyperplanes} 
P.~Orlik, L.~Solomon, and H.~Terao, 
\emph{Arrangements of hyperplanes and differential forms}. 
Combinatorics and algebra (Boulder, Colo., 1983), 29--65,
Contemp. Math., \textbf{34}, Amer. Math. Soc., Providence, RI, 1984. 

\bibitem[OT92]{orlikterao:arrangements} P.~Orlik and H.~Terao,
  \emph{Arrangements of hyperplanes}, Springer-Verlag, 1992.

\bibitem[S{\etalchar{+}}97]{gap3} M. Sch{\accent127 o}nert et~al.,
  \emph{{GAP} -- {Groups}, {Algorithms}, and {Programming} -- version 3
   release 4}, 
1997.

\bibitem[ST54]{shephardtodd}
G.C. Shephard and J.A. Todd, 
\emph{Finite unitary reflection groups}.
Canadian J. Math. \textbf{6}, (1954), 274--304. 

\bibitem[S{\etalchar{+}}09]{sage}
W.\thinspace{}A. Stein et~al., \emph{{S}age {M}athematics {S}oftware}, 
The Sage Development Team, 2009, 
\url{http://www.sagemath.org}.

\bibitem[Sta72]{stanley:super}
R. P. Stanley, 
\emph{Supersolvable lattices},
Algebra Universalis \textbf{2} (1972), 197--217. 

\bibitem[Ste60]{steinberg:invariants}
R. Steinberg, 
\emph{Invariants of finite reflection groups},
Canad. J. Math. \textbf{12}, (1960), 616--618.

\bibitem[Ter80]{terao:freeI} H.~Terao, \emph{Arrangements of hyperplanes and
    their freeness I}, J. Fac. Sci.  Univ. Tokyo \textbf{27} (1980), 293--320.

\bibitem[Ter81]{terao:freefactors}
\bysame,
\emph{Generalized exponents of a free arrangement of hyperplanes and 
Shepherd-Todd-Brieskorn formula}, Invent. Math. \textbf{63} no.\ 1 (1981) 159--179.

\bibitem[Ter92]{terao:factored}
\bysame,
\emph{Factorizations of the Orlik-Solomon Algebras},
Adv. in Math. \textbf{92}, (1992), 45--53.

\end{thebibliography}

\newcommand{\etalchar}[1]{$^{#1}$}
\providecommand{\bysame}{\leavevmode\hbox to3em{\hrulefill}\thinspace}
\providecommand{\MR}{\relax\ifhmode\unskip\space\fi MR }
\providecommand{\MRhref}[2]{%
  \href{http://www.ams.org/mathscinet-getitem?mr=#1}{#2} }
\providecommand{\href}[2]{#2}


\end{document}